\documentclass[a4 paper,11pt]{article}

\usepackage{epsfig}
\usepackage{amsmath}
\usepackage[latin1]{inputenc}
\usepackage{amsmath,amsthm,amsfonts,amssymb,amscd}
\usepackage{graphicx}
\usepackage{srcltx}

\theoremstyle{plain}
\newtheorem{theorem}{Theorem}[section]
\newtheorem{proposition}[theorem]{Proposition}
\newtheorem{lemma}[theorem]{Lemma}
\newtheorem{corollary}[theorem]{Corollary}

\theoremstyle{definition} \theoremstyle{remark}
\newtheorem{remark}[theorem]{Remark}

\newtheorem{definition}[theorem]{Definition}

\newtheorem{question}{Question}

    \newtheoremstyle{TheoremNum}
        {\topsep}{\topsep}              
        {\itshape}                      
        {}                              
        {\bfseries}                     
        {.}                             
        { }                             
        {\thmname{#1}\thmnote{ \bfseries #3}}
    \theoremstyle{TheoremNum}
    \newtheorem{thmn}{Theorem}
		
		\newtheorem{coron}{Corollary}

\begin{document}

\large

\begin{center}
\textbf{\large Partially hyperbolic geodesic flows}

\vspace{10 mm}  

{\large by Fernando Carneiro and Enrique Pujals}

\end{center}

\begin{abstract}
   We construct a category of examples of partially hyperbolic geodesic flows which are not Anosov, deforming the metric of a compact locally symmetric space of nonconstant negative curvature. Candidates for such an example as the product metric and locally symmetric spaces of nonpositive curvature with rank bigger than one are not partially hyperbolic. We prove that if a metric of nonpositive curvature has a partially hyperbolic geodesic flow, then its rank is one. Other obstructions to partial hyperbolicity of a geodesic flow are also analyzed.
\end{abstract}

\tableofcontents

\section{Introduction}
\label{s.intro}

The theory of  hyperbolic dynamics has been one of the extremely successful stories in dynamical systems. Originated by studying dynamical properties of geodesic flows on
manifolds with negative curvature \cite{An} and geometrical properties of homoclinic points
\cite{Sm},  hyperbolicity  is the cornerstone of uniform and robust chaotic dynamics; it  characterizes the structural stable systems; it provides the structure underlying the presence of homoclinic points; a large category of rich dynamics are hyperbolic (geodesic flows in negative curvature, billiards with negative curvature, linear automorphisms, some mechanical systems, etc.); the hyperbolic theory has been fruitful in developing a geometrical approach to dynamical systems; and, under the assumption of hyperbolicity one obtains a satisfactory (complete) description of the dynamics of the  system from a topological and statistical point of view. Moreover, hyperbolicity has provided paradigms or models of behavior that can be expected to be obtained in specific problems.

Nevertheless,  hyperbolicity was soon realized to be a property less universal than it was initially thought: it was shown that there are open sets in the space of dynamics which
are nonhyperbolic. To overcome these difficulties, the theory moved in different directions; one being  to develop weaker or relaxed forms of hyperbolicity, hoping to include a larger
class of dynamics.

There is an easy  way to relax hyperbolicity, called partial hyperbolicity. Let $f : M \to M$ be a diffeomorphism from a smooth manifold $M$ to itself. We say that $f$ is partially hyperbolic if the tangent bundle of $M$ split into $Df$-invariant subbundles $TM=E^s\oplus E^c\oplus E^u,$  such that the behavior of vectors in $E^s, E^u$ are contracted and expanded respectively by $Df$, but vectors in $E^c$ may be neutral for the action of the tangent map, i.e., $|d_xf^n v|$ contracts exponentially fast if $v \in E^s$, $|d_xf^n v|$ expands exponentially fast if $v \in E^u$ and $|d_xf^n v|$ neither contracts nor expands as fast as for the other two invariant subbundles if $v \in E^c$. The Anosov condition is equivalent to $E^c(x) = \{0 \}$ for all $x \in M$. This notion arose in a natural way   in the context of time one maps of Anosov flows, frame flows or group extensions. See \cite{BP}, \cite{Sh}, \cite{M}, \cite{BD}, \cite{BV} for examples of these systems  and \cite{HP}, \cite{PS} for an overview.

However, and differently from hyperbolic ones, partially hyperbolic non-Anosov systems were unknown in the context of geodesic flows induced by Riemannian metrics. As far as we know, the way to produce partially hyperbolic systems in discrete dynamics are the following: time-one maps of Anosov flows, skew-products over hyperbolic dynamics, products and derived of Anosov deformations (DA). The two last approaches can be adapted to flows.

Our work shows that one is able to deform a specific metric that provides an Anosov geodesic flow to get a partially hyperbolic geodesic flow. Theorem A is inspired by the Ma\~n\'e's DA construction of a partially hyperbolic diffeomorphism \cite{M}.

We prove the following theorems:

\begin{thmn}[A]
\label{t.a}
There are Riemannian metrics such that their geodesic flows are partially hyperbolic but not Anosov. Moreover, some of them are transitive.
\end{thmn}

More precisely, we prove:

\begin{thmn}[B]
\label{t.b}
Let $(M,g)$ be a compact locally symmetric Riemannian manifold of nonconstant negative sectional curvature. There is an $C^2$-open set of $C^{\infty}$ Riemannian metrics on $M$ such that their geodesic flows are partially hyperbolic but not Anosov. Those metrics which are on the $C^2$-boundary of the Anosov Riemannian metrics are transitive.
\end{thmn}

\begin{remark}
These metrics are $C^1$-close and $C^2$-far from $g$, and some are $C^2$-far from Anosov.
\end{remark}

\begin{remark}
The theorem works for the compact K\"ahler manifold of constant holomorphic curvature $-1$ \cite{G}, and also for the quaternionic K\"ahler locally symmetric spaces of negative curvature. Both of these locally symmetric spaces are even-dimensional. In these cases, for a fixed $v \in TM$, there are subspaces $A$ and $B$ of $v^{\bot}$ such that if $w \in A$ then $K(v,w) = -1$, if $w \in B$ then $K(v,w) = -\frac{1}{4}$ and $v^{\bot} = A \oplus B$. This property implies that the geodesic flow of these manifolds is Anosov with many invariant subbundles (see section \ref{s.symm}).

\end{remark}

\begin{remark}
Of course, if we multiply the metric by a constant, the Anosov or the partially hyperbolic splitting remain the same, but the curvature does not. So, we consider the maximal sectional curvature of the locally symmetric space to be $-1$, which is true after multiplication of the metric by a constant.
\end{remark}

\begin{remark}
A classical Ma\~n\'e theorem \cite{M3} says that if, for a geodesic flow of a Riemannian manifold there is an invariant Lagrangian subbundle, then this Riemannian manifold does not have conjugate points. The existence of a partially hyperbolic non-Anosov geodesic flow implies that this theorem does not extend to invariant isotropic subbundles. Corollary C.1  states that some partially hyperbolic non-Anosov do not have conjugate points. 
\end{remark}

Let $X: N \to TN$ be a vector field on $N$ without singularities, i.e., for any $p \in N$, $X(p) \neq 0$, let $\phi_t: N \to N$ be its flow, let $p \in N$ be a point such that $\phi_T(p) = p$ for some positive real number $T$, and let $\lambda_i$, $i=1,\ldots,dim(N)$ be the eigenvalues of $d_p\phi_T : T_pN \to T_pN$. Let $\lambda_1 = 1$ be the eigenvalue associated to $X(p)$. We say $p$ is hyperbolic if there is no eigenvalue in the unit circle, besides $\lambda_1$, i.e., $\lambda_i \neq 1$ for $i=2,\ldots,dim(N)$. We say $p$ is quasi-elliptic if there are eigenvalues in the unit circle, besides $\lambda_1$. We say $p$ is nondegenerate if there is no eigenvalue equal to one besides $\lambda_1$. The next two corollaries are given by the persistence of quasi-elliptic nondegenerate periodic points.

\begin{coron}[C.1]	
There is a $C^2$-open set $\mathcal{U}$ of metrics in the set of metrics of $M$ such that for $g \in \mathcal{U}$, the geodesic flow of $g$ is partially hyperbolic but not Anosov, for $(M,g)$ as in the previous theorem. There is also an open set $\mathcal{U}'$ of metrics such that for $g \in \mathcal{U}'$, the geodesic flow of $g$ is partially hyperbolic non-Anosov and with conjugate points.
\end{coron}

\begin{coron}[C.2]
There is a $C^2$-open set $\mathcal{V}$ of Hamiltonians in the set of Hamiltonians of $(TM,\omega_{TM})$, near geodesic Hamiltonians, such that for $h \in \mathcal{U}$, the Hamiltonian flow of $h$ is partially hyperbolic but not Anosov.
\end{coron}

It is easy to construct  partially hyperbolic Hamiltonians by suspensions; but they are not close to geodesic flows. 

We also show that product metrics of Anosov geodesic flows are not examples with the partially hyperbolic property:

\begin{thmn}[D]
\label{t.d}
If $(M_1,g^1)$ and $(M_2,g^2)$ are Riemannian manifolds such that the geodesic flow of at least one of them is Anosov, then the geodesic flow of $(M_1 \times M_2,g^1+g^2)$ is not partially hyperbolic.
\end{thmn}

For compact locally symmetric spaces of nonpositive curvature the following holds:

\begin{thmn}[E]
\label{t.e}
If the geodesic flow of a compact locally symmetric space of nonpositive curvature is partially hyperbolic, then its geodesic flow is Anosov.
\end{thmn}

The proof of theorem D and theorem E imply the following:

\begin{thmn}[F]
If $(M,g)$ is a compact Riemannian manifold with nonpositive sectional curvature and partially hyperbolic geodesic flow then $(M,g)$ has rank one.
\end{thmn}

Moreover, 

\begin{thmn}[G]
If $(M^n,g)$ is a Riemannian manifold with partially hyperbolic geodesic flow then $n$ is even, and if $n \equiv 2 \mod 4$, then $\dim E^s = 1$ or $n-1$.
\end{thmn}

Roughly speaking, the strategy of the proof of theorem A  follows these steps:

\begin{itemize}
\item[1.] We chose a metric whose geodesic flow is Anosov and whose hyperbolic invariant splitting is of the form $T(UM) = E^{ss} \oplus E^s \oplus \left\langle X \right\rangle \oplus E^u \oplus E^{uu}$ (section \ref{s.orig}); 
\item[2.] We take a closed geodesic $\gamma$ without self-intersections (section \ref{s.prop});
\item[3.] We change the metric in a tubular neighborhood of $\gamma$ in $M$, such that along $\gamma$ the strong subbundles ($E^{ss}$ and $E^{uu}$) remain invariant and the weak subbundles disappear, becoming a central subbundle with no hyperbolic behavior (section \ref{s.geode}):
\item[4.] Outside the tubular neighborhood of $\gamma$, the dynamics remains hyperbolic;
\item[5.] We show that for the geodesics that intersect the tubular neighborhood the cones associated to the extremal subbundles ($E^{ss}$ and $E^{uu}$)  are preserved (sections \ref{s.para1}, \ref{s.para2}, \ref{s.transv}, \ref{s.phcone}); 
\item[6.] We prove that for vectors in the unstable cones there is expansion, and for vectors inside the stable cones there is contraction, under the action of the derivative of the new geodesic flow (section \ref{s.expo}).
\end{itemize}

\begin{remark}
In the case of the geodesic flow of a compact locally symmetric Riemannian manifold $(M,g)$ of nonconstant negative curvature, the tangent bundle of the unitary tangent bundle $UM$ splits in many invariant subbundles: $T(UM) = E^{ss} \oplus E^{s} \oplus \langle X \rangle \oplus E^u \oplus E^{uu}$. If $\theta \in UM$, $\zeta \in E^{ss}(\theta)$ then $|d_{\theta} \zeta|$ contracts exponentially fast as $t \rightarrow \infty$, faster than if $\zeta \in E^{s}$. If $\zeta \in E^{uu}$ then $|d_{\theta} \zeta|$ expands exponentially fast as $t \rightarrow \infty$, faster than if $\zeta \in E^{u}$ (see section \ref{s.orig}). 
\end{remark}

We would like to recall that in the symplectic context, the existence of  dominated splitting with two subbundles of equal dimension implies hyperbolicity. This was first observed by Newhouse for surfaces maps \cite{Ne}, later by  Ma\~n\'e in any dimension \cite{M2} for symplectic maps, by Ruggiero in the context of geodesic flows \cite{R} and Contreras for symplectic and contact flows \cite{Co1}. We want to point out that these results do not contradict ours: the splitting for the examples of theorem \ref{t.a} and theorem \ref{t.b} contain more than two invariant subbundles.

There are partially  hyperbolic $\Sigma$-geodesic flows, defined over a distribution $\Sigma \subsetneqq TM$ which arise in the study of the dynamics of free particles in a system with
constraints (see  \cite{CKO}). However, if the distribution is involutive then the leaves of the distribution have negative curvature, and we are again in the Anosov geodesic flows case. If the distribution is not involutive, the $\Sigma$-geodesic flow is not a geodesic flow.

The article is organized as follows:

In the second section of the article, we introduce basic results about  geodesic flows, partial hyperbolicity and the equivalent property of the proper invariance of cone fields \cite{P}, \cite{HP}. 

In the third section we prove theorem D. We show that product metrics are not examples of partially hyperbolic non-Anosov geodesic flows.

In the fourth section we introduce properties and the classification of locally symmetric spaces of negative curvature which are the natural candidates to deform into partially hyperbolic non-Anosov geodesic flows.

In the fifth section we prove theorem B, and so theorem A, except for the transitivity of some of the examples. We show that the deformed metric has a partial hyperbolic non-Anosov geodesic flow. We give a proof of the proper invariance of the strong cones, i.e., the derivative of the geodesic flow brings the strong unstable cone field properly inside itself, and the derivative of the inverse of the geodesic flow brings the stable cone field properly inside itself. The proof is based on the calculation of the variation of the opening of the cones of an appropriate cone field, and then we prove the exponential expansion or contraction for vectors in the strong unstable and stable cones.

In the sixth section we prove theorem E. We show that compact locally symmetric spaces of nonpositive sectional curvature are not partially hyperbolic, i.e., they have dominated splitting with many invariant subbundles, except the spaces of nonconstant negative sectional curvature. Locally symmetric spaces of constant negative curvature are Anosov with only two invariant subbundles. Locally symmetric spaces of nonconstant negative curvature are Anosov with at least four invariant subbundles in their dominated splitting, so they are the candidates for the deformation, since they have the property mentioned in the first item of the strategy above  \cite{E2},\cite{E3}, \cite{J}.

In the last section we prove theorems F and G and the transitivity of some of the examples, which is stated in theorems A and B. We show some obstructions to the existence of a partially hyperbolic geodesic flow. In general, there are obstructions for the rank of the manifold if the Riemannian manifold has nonpositive sectional curvature, and for the dimension of the Riemannian manifold and the dimension of the hyperbolic invariant subbundles.

\vspace{5 mm}

\textbf{\large Acknowledgements:}

\vspace{5 mm}

The first author would like to thank Luis Florit, Marcelo Viana, Rafael Ruggiero, Leonardo Macarini, Henrique Bursztyn, Umberto Hryniewicz, Jacob Palis Jr., Fernando Cod\'a, Alexander Arbieto, Gonzalo Contreras, Wolfgang Ziller for fruitful conversations concerning this work. We would like to thank Gabriel Paternain who presented us the compact K\"ahler manifold of constant holomorphic curvature $-1$ and made suggestions after we finished a first draft of the article. Also the first author would like to thank Rafael Ruggiero for the suggestion to study the case of locally symmetric spaces, a more general case than the one studied in the first author's thesis. Through this work the first author also thanks the financial support granted by CNPq through his doctorate under the advice of Enrique Pujals at Impa, and the support granted by Capes for his post-doctorate under the advice of Alexander Arbieto at UFRJ.

\section{Preliminaries}
\label{s.predef}

In this section, we give some preliminary definitions. In the first subsection, the definitions are about geodesic flows. The basic reference for this subsection is the book by Paternain \cite{P}.  In the second subsection, we give the main definitions about partial hyperbolicity and the basic reference is the survey by Hasselblat and Pesin \cite{HP}. 

\subsection{Geodesic flows}

A Riemannian manifold $(M,g)$ is a $C^\infty$-manifold with an Euclidean inner product $g_x$ in each $T_xM$ which varies smoothly with respect to $x \in M$. So a Riemannian metric is a smooth section $g: M \to Symm_2^+(TM)$, where $Symm_2^+(TM)$ is the set of positive definite bilinear and symmetric forms in $TM$. Along the article we will consider the topology of the space of metrics of a manifold $M$ to be the $C^2$-topology on the space of these sections. 

The geodesic flow of the metric $g$ is the flow 
$$\phi_t: TM \to TM ,$$ $$\phi_t(x,v) = (\gamma_{(x,v)}(t),\gamma'_{(x,v)}(t)),$$
\noindent such that $\gamma_{(x,v)}$ is the geodesic for the metric $g$ with initial conditions $\gamma_{(x,v)}(0) = x$ and $\gamma'_{(x,v)}(0) = v$, $x \in M$, $v \in T_xM$. Since the speed of the geodesics is constant, we can consider the flow restricted to $UM := \{ (x,v) \in TM : g_x(v,v) = 1 \}$.

\begin{definition}
Let  $\pi_M: TM \to M$ be the canonical projection of the tangent bundle. The vertical subbundle, $\pi_V: V(TM) \to TM$, is the bundle whose fiber at $\theta \in T_xM$ is given by $V(\theta) = ker(d_{\theta}\pi_M)$.
\end{definition}

\begin{definition}\label{def.basic}
$K: T(TM) \to TM$, which is called the connection map associated to the metric $g$, is defined as follows: given $\xi \in T_{\theta}TM$ let $z: (-\epsilon,\epsilon) \to TM$ be an adapted curve to $\xi$; let $\alpha: (-\epsilon,\epsilon) \to M: t \to \pi_M \circ z (t)$, and $Z$ the vector field along $\alpha$ such that $z(t) = (\alpha(t),Z(t))$; then $K_{\theta}(\xi) := (\nabla_{\alpha'}Z)(0)$. $\pi_H: H(TM) \to TM$, the horizontal subbundle, is given by $H(\theta) := ker(K_{\theta})$.
\end{definition}

Some properties of $H$ and $V$ are: 
\begin{itemize}
\item[1.]$H(\theta) \cap V(\theta) = 0$, 
\item[2.]$d_{\theta}\pi$ and $K_{\theta}$ give identifications of $H(\theta)$ and $V(\theta)$ with $T_xM$, 
\item[3.]$T_{\theta}TM = H(\theta) \oplus V(\theta)$.
\end{itemize}

The geodesic vector field $G: TM \to T(TM)$ is given by $G(v)=(v,0)$ in the decomposition $H \oplus V \approx \pi^*TM \oplus \pi^*TM$.

The decomposition in horizontal and vertical subbundles allows us to define the Sasaki metric on $TM$:
\begin{eqnarray*} \widehat{g}_{\theta}(\xi,\eta) & := & g_x(d_{\theta}\pi(\xi),d_{\theta}\pi(\eta)) + g_x(K_{\theta}(\xi),K_{\theta}(\eta)) \\ & = & g_x(\xi_h,\eta_h) + g_x(\xi_v,\eta_v) \end{eqnarray*}
\noindent for $\xi$ and $\eta \in T_{\theta}TM$, with $\xi = (\xi_h,\xi_v)$ and $\eta = (\eta_h,\eta_v)$ in the decomposition $T_{\theta}TM = H(\theta) \oplus V(\theta)$, with $\xi_h$ and $\eta_h \in T_xM \cong H(\theta)$, $\xi_v$ and $\eta_v \in T_xM \cong V(\theta)$.

It also allows us to define a symplectic $2$-form and an almost complex structure $\widetilde{J}$ on $TM$ and a contact form on $UM$:

\begin{eqnarray*} \Omega_{\theta}(\xi,\eta) & := & g_x(d_{\theta}\pi(\xi),K_{\theta}(\eta)) - g_x(K_{\theta}(\xi),d_{\theta}\pi(\eta)) \\ & = & g_x(\xi_h,\eta_v) - g_x(\eta_h,\xi_v), \\
\widetilde{J}(\xi_h,\xi_v) & := & (-\xi_v,\xi_h), \\
\alpha_{\theta}(\xi) & := & \widehat{g}_{\theta}(\xi,G(\theta)) = g_x(d_{\theta}\pi(\xi),v) = g_x(\xi_h,v).
\end{eqnarray*}

\begin{definition}
The geodesic flow leaves $UM$ invariant, and we can define a contact form on $UM$ such that its Reeb vector field is the geodesic vector field: $\pi_S: S(UM) \to UM$ is the contact structure bundle on $UM$, with fiber $S(\theta) := ker(\alpha_{\theta})$. It is an invariant subbundle for the geodesic flow, and $\mathbb{R} \cdot G \oplus S = T(UM)$. The vector bundle $S(UM)$ also has a decomposition on horizontal and vertical subbundles.
\end{definition}

In the next definition, we relate the derivative of the geodesic flow with the Jacobi fields of the metric that generates the flow.

\begin{definition}
Let $\gamma_{\theta}$ be a geodesic on the Riemannian manifold $(M,g)$ with initial conditions $\gamma_{\theta}(0) = x$, $\gamma'_{\theta}(0) = v$, where $\theta = (x,v)$, $x \in M$, $v \in T_xM$. A Jacobi field $\zeta(t)$ along a geodesic $\gamma_{\theta}$ is a vector field obtained by a variation of the geodesic $\gamma_{\theta}$ through geodesics: 
$$\zeta(t) := \frac{\partial}{\partial s}|_{s=0} \pi \circ \phi_t(z(s)),$$ 
\noindent where $z(0) = \theta$, $z'(0) = \xi$ and $z(s) = (\alpha(s),Z(s))$ (where $z, \alpha, Z$ were introduced in definition \ref{def.basic}).
It satisfies the following equation:
$$\zeta'' + R(\gamma'_{\theta},\zeta)\gamma'_{\theta} = 0.$$
\noindent Its initial conditions are:
$$\zeta(0) = \frac{\partial}{\partial s}|_{s=0} \pi \circ z(s) = d_{\theta}\pi \xi = \xi_h,$$
\begin{eqnarray*} \zeta'(0) & = & \frac{D}{dt} \frac{\partial}{\partial s}|_{t=0,s=0} \pi \circ \phi_t (z(s)) = \frac{D}{\partial s} \frac{\partial}{\partial t}|_{s=0,t=0}  \pi \circ \phi_t (z(s)) \\ & = & \frac{D}{\partial s} |_{s=0} Z(s) = K_{\theta} \xi = \xi_v. \end{eqnarray*}

\end{definition}
\noindent Observe that the choice of $\xi$  determines a Jacobi field $\zeta_\xi$  along $\gamma_\theta$ and so the derivative of a geodesic flow is: $d_{\theta}\phi_t(\xi) = ( \zeta_{\xi}(t), \zeta'_{\xi}(t) )$.

We can restrict the action of the derivative of the geodesic flow to the contact structure: $d_{\theta}\phi_t(\xi) : S(UM) \to S(UM)$. In this case the Jacobi fields associated with the contact structure are the orthogonal to the geodesics on $M$: $\zeta: \mathbb{R} \to \gamma_{\theta}^*TM$ such that $\zeta(0) = \xi_h$, $\zeta'(0) = \xi_v$, where $\xi \in T_{\theta}TM$, $\xi_h \bot \theta$, $\xi_v \bot \theta$.

\begin{remark}
We define the curvature tensor $R: \Gamma(TM) \times \Gamma(TM) \times \Gamma(TM) \to \Gamma(TM)$ as it is done in do Carmo's book \cite{Ca}:

$$R(X,Y)Z := \nabla_Y \nabla_X Z - \nabla_X \nabla_Y Z + \nabla_{[X,Y]} Z.$$
\end{remark}

\subsection{Partial hyperbolicity}
\label{ss.ph}

\begin{definition}
A partially hyperbolic flow $\phi_t : N \to N$ in the manifold $N$ generated by the vector field $X: N \to TN$ is a flow such that its quotient bundle $TN / \langle X \rangle$ has an invariant splitting $TN / \langle X \rangle = E^s \oplus E^c \oplus E^u$ such that these subbundles are non-trivial and have the following properties:
$$ d_x\phi_t (E^s(x)) = E^s(\phi_t(x)), $$ $$d_x\phi_t (E^c(x)) = E^c(\phi_t(x)),$$ $$d_x\phi_t (E^u(x)) = E^u(\phi_t(x)), $$
$$ || d_x\phi_t |_{E^s} || \leq C \exp(t \lambda),$$ $$|| d_x\phi_{-t} |_{E^u} || \leq C \exp(t \lambda), $$ $$C \exp(t \mu) \leq || d_x\phi_t |_{E^c} || \leq C \exp(- t \mu),$$
\noindent for some constants $\lambda < \mu < 0 < C$ and for all $x \in N$. In the context of the present paper, $N$ is the unitarian tangent bundle $SM$.
\end{definition}  

\begin{definition}
A splitting $E \oplus F$ of the quotient bundle $TN / \langle X \rangle$ is called a dominated splitting if:
$$ d_x\phi_t (E(x)) = E(\phi_t(x)),$$ $$ d_x\phi_t (F(x)) = F(\phi_t(x)), $$
$$ || d_x\phi_t |_{E(x)} || \cdot || d_{\phi_t(x)} \phi_{-t} |_{F(\phi_t(x))} || < C \exp(- t \lambda)$$
\noindent for some constants $C$, $\lambda > 0$ and $x \in N$.

We only need to prove the existence of dominated splitting because of well known result in conservative dynamics \cite{Co1}, \cite{R}:

\begin{lemma}
\label{l.contreras}
Let $N$ be a smooth manifold and $\phi_t: N \to N$ a flow on $N$ such that there is a dominated splitting $TN / \langle X \rangle = E^{cs} \oplus E^c \oplus E^{cu}$, then for all $x \in N$ there are positive real number $C$ and $\lambda$ such that $$ || d_x\phi_t |_{E^{cs}} || \leq C \exp(t \lambda),$$ $$|| d_x\phi_{-t} |_{E^{cu}} || \leq C \exp(t \lambda).$$
\end{lemma}

\end{definition}

\subsubsection{Partial hyperbolicity and cone fields}\label{s.ph and cones}

There is a useful criterion  for verifying partial hyperbolicity, called the cone criterion:

Given $x \in N$, a subspace $E(x) \subset T_xN$ and a number $\delta$, we define the cone at $x$ centered around $E(x)$ with angle $\delta$ as
$$C(x,E(x),\delta) = \{ v \in T_xN : \angle(v,E(x)) < \delta \},$$
\noindent where $\angle(v,E(x))$ is the angle that the vector $v \in T_xN$ makes with its own projection to the subspace $E(x) \subset T_xN$. Sometimes, the constant $\delta$ involves in the definition of the cone $C(x,E(x),\delta)$ is called the {\it opening of the cone}.

A flow is partially hyperbolic if there are $\delta > 0$, some time $T > 0$, and two continuous cone families $C(x,E_1(x),\delta)$ and $C(x,E_2(x),\delta)$ such that:
$$d_x\phi_{-t}(C(x,E_1(x),\delta)) \subsetneqq C(x,E_1(\phi_{-t}(x)),\delta),$$
$$d_x\phi_{t}(C(x,E_2(x),\delta)) \subsetneqq C(x,E_2(\phi_{t}(x)),\delta),$$
$$\|d_x\phi_t \xi_1 \| < K \exp(t \lambda),$$ $$ \|d_x\phi_{-t} \xi_2 \| < K \exp(t \lambda),$$
\noindent for $\xi_1 \in C(x,E_1(x),\delta)$, $\xi_2 \in C(x,E_2(\phi_t(x)),\delta)$, some constants $K > 0$, $\lambda < 0$ and all $t > 0$.

\subsubsection{Partial hyperbolicity and angle cone variation}
\label{s. cone variation def}
Let $Pr_E: TN \to E$ be the orthogonal projection to $E$, where $\pi_E: E \to N$ is a vector subbundle of $TN$. We define 
\begin{eqnarray}
\label{eq.cone0}
\Theta(v,E) := \frac{g(Pr_E v,Pr_E v)}{g(v,v)}.
\end{eqnarray}
To prove that  the  proper invariance of cones holds, it is enough  to check the following inequality:
\begin{eqnarray}
\label{eq.cone}
\frac{d}{dt} \Theta(d_x\phi_t(v),E(\phi_t(x))) > 0
\end{eqnarray}
\noindent for $v \in \partial C(x,E(x),\delta) := \{w \in T_xN : \angle(w,E(x)) = \delta\}$. 

\begin{remark}
The quantity $\Theta(v,E)$ equals twice the square of the cosine of the angle between $v$ and the subspace $E$, so if it increases along the flow, then the cone field is properly invariant.  We call the derivative above the {\it angle cone variation}.  This calculation is inspired by the calculations in \cite{W}, although we do not use quadratic forms here.
\end{remark}

The proper invariance of the cones by the derivative of the geodesic flow implies the existence of a dominated splitting. For the exponential expansion or contraction in the unstable and stable directions, respectively, we only need to check exponential expansion or contraction inside the unstable and stable cones, respectively.

\begin{lemma}
\label{l.angle}
For a fixed $\delta > 0$, and a fixed subbundle $E \to N$, $E(x) \subset T_xN$, if inequality (\ref{eq.cone}) holds for $v \in \partial C(x,E(x),\delta)$, then the cone field is proper invariant for the geodesic flow.
\end{lemma}

\begin{proof}
Let $c \in (1,2)$ be such that $2\cos^2(\delta) = c$. Then 
\begin{eqnarray*}
C(x,E(x),\delta) & = & \left\{ \Theta(v,E) \geq c \right\}, \\ \partial C(x,E(x),\delta) & = & \left\{ \Theta(v,E) = c \right\}.
\end{eqnarray*} 
Notice that the quantity on the left side of (\ref{eq.cone}) is the same for $v$ and for $kv$ for every $k > 0$. Then we can calculate for $v$ such that $g(v,v) = 1$. Define
$$\partial_1 C(x,E(x),\delta) = \{ g(Pr_E v,Pr_E v) = c, g(v,v) = 1 \}.$$
Then the set of vectors in the  above defined boundary of the cones  is compact, which implies that its derivative is bounded away from zero:
$$ \frac{d}{dt} \Theta(d_x\phi_t(v),E(\phi_t(x))) \geq a > 0. $$
Its immediate consequence is that the cone field is properly invariant by the flow.
\end{proof}

\section{The geodesic flow of a product metric is not partially hyperbolic}
\label{s.whynot}

Now, we are going to show that some simple candidates for partially hyperbolic geodesic flows are not partially hyperbolic. In particular, we are going to prove that product metrics are neither Anosov nor partially hyperbolic.

A natural candidate for symplectic partially hyperbolic dynamics is the following: for any hyperbolic symplectic action $\Phi : \mathbb{R} \to Sp(E,\omega)$, $\pi: E \to B$ a symplectic bundle with $\omega$ as its symplectic $2$-form, one can produce another symplectic action $\Phi^*: \mathbb{R} \to Sp(E,\omega) \oplus Sp(B \times \mathbb{R}^2,\omega_0) : t \to \Phi(t) \oplus Id$, where $\omega_0 = dx \wedge dy$ is the canonical symplectic form of $\mathbb{R}^2 = \{(x,y): x,y \in \mathbb{R} \}$. The symplectic flow associated with this symplectic $\mathbb{R}$-action is partially hyperbolic with a central direction of dimension $2$. For geodesic flows the above described construction does not work.

\begin{theorem} Let $(M,g)$ be a Riemannian manifold  whose geodesic flow is Anosov.
Then, the product Riemannian manifold $(M \times \mathbb{T}^n, g + g_0)$ where $(\mathbb{T}^n,g_0)$ is $\mathbb{T}^n$ with its canonical flat metric, is not partially hyperbolic.
\end{theorem}

\begin{proof}
Notice that for any $x \in M$, $\{x\} \times \mathbb{T}^n$ is a totally geodesic submanifold of $(M \times \mathbb{T}^n, g + g^0)$. So, its second fundamental form is identically zero. Since the metric in $\mathbb{T}^n$ is flat this implies that, for any $x \in M$ and any $y \in \mathbb{T}^n$:
$$R(\gamma'_{(x,y,0,v)},(0,w))\gamma'_{(x,y,0,v)} = 0,$$ where $\gamma_{(x,y,0,v)}$ is the geodesic of $(M \times \mathbb{T}^n, g + g^0)$ starting at $(x,y) \in M \times \mathbb{T}^n$ with 
 $\gamma'_{(x,y,0,v)}(0) = (0,v) \in T_xM \times \mathbb{R}^n$, the covariant derivative of the geodesic starting
at $(x, y)$ with tangent vector $(0, v)$.

For a product metric in $(M_1 \times M_2,g^1+g^2)$, let us say $R$ is the curvature tensor of the product Riemannian manifold with the product metric and  $R^1$ the curvature tensor of the Riemannian manifold $M_1$. Then the following properties hold:

\begin{itemize}
\item[i.] $R(X,Y,Z,W) = R^1(X,Y,Z,W)$, for $X,Y,Z,W$ tangent to $M_1$, because of the Gauss' equation and the fact that the second fundamental form is zero \cite{Ca};
\item[ii.] $R(X,Y,Z,N) = 0$, for $X,Y,Z$ tangent to $M_1$ and $N$ tangent to $M_2$, because of Codazzi's equation and the fact that the second fundamental form is zero \cite{Ca};
\item[iii.] $R(X,N,X,\widehat{N}) = 0$, for $X,Y$ tangent to $M_1$ and $N,\widehat{N}$ tangent to $M_2$, because $K(X,N) = 0$ \cite{Ca}.
\end{itemize}

Then, for a submanifold $\{x\} \times \mathbb{T}^n$ with the flat metric:
$$R(\gamma'_{(x,y,0,v)},\cdot)\gamma'_{(x,y,0,v)} \equiv 0.$$
\noindent So, the derivative of the geodesic flow along geodesics in $\{x\} \times \mathbb{T}^n$ does not have any exponential contraction or expansion. So, there is no partially hyperbolic splitting for its geodesic flow.
\end{proof}

\begin{theorem} Let $(M_1,g^1)$ and $(M_2,g^2)$ be two Riemannian manifolds whose  geodesic flows are Anosov.
The geodesic flow of the Riemannian manifold $(M_1 \times M_2, g^1 + g^2)$ is not Anosov.
\end{theorem}

\begin{proof}
The proof that this geodesic flow is not Anosov is easy. It is a classical result that $(x_0, \gamma_{(y,v)}(t))$ and $(\gamma_{(x,u)}(t),y_0)$ are geodesics of the product metric, $x_0 \in M_1$, $y_0 \in M_2$, $u \in T_xM_1$, $v \in T_yM_2$, $\gamma_{(x,u)}(0) = x$ and $\gamma'_{(x,u)}(0) = u$, $\gamma_{(y,v)}(0) = y$ and $\gamma'_{(y,v)}(0)=v$. So, we choose $x_0$ and $x_1 \in M_1$ close enough. Then $(x_0, \gamma_{(y,v)}(t))$ and $(x_1, \gamma_{(y,v)}(t))$ are two geodesics of the product metric with initial conditions $(x_0,y,0,v)$ and $(x_1,y,0,v)$. Let $dist$ be the distance function for the Sasaki metric of $U(M_1 \times M_2)$ and $dist_1$ be the distance function for the Sasaki metric of $UM_1$. The geodesic flow is not expansive, because $dist(\phi_t(x_0,y,0,v),\phi_t(x_1,y,0,v)) = dist_1((x_0,0),(x_1,0))$ and therefore, for  any $\epsilon > 0$, if $x_0$ and $x_1$ are close enough, then $dist_1((x_0,0),(x_1,0)) < \epsilon$, and so the geodesic flow is not Anosov.
\end{proof}

\begin{theorem}
\label{t.prod}
The geodesic flow of the product metric of a product manifold of two Riemannian manifolds with Anosov geodesic flows is not partially hyperbolic.
\end{theorem}

\begin{proof}

Take local coordinates for the geodesic flow of the product metric. Let $x \in M_1$, $y \in M_2$, $u \in T_xM_1$, $v \in T_yM_2$, and let $\gamma_{(x,y,u,v)}(t)$ be the geodesic with initial conditions $\gamma_{(x,y,u,v)}(0) = (x,y)$ and $\gamma'_{(x,y,u,v)}(0) = (u,v)$. Since the product metric is a sum of the two metrics, we have that $\pi_i: M_1 \times M_2 \to M_i$, $i = 1, 2$, the natural projection from the product manifold to $M_i$, is a isometric submersion. So $\gamma_{(x,y,u,v)}(t) = (\gamma_{(x,u)}(t),\gamma_{(y,v)}(t))$. 

Let us construct an orthonormal basis of parallel vector fields for $\gamma_{(x,y,u,v)}(t)$. 
Suppose $g^1_x(u,u) = 1$ and $g^2_y(v,v) = 1$. So, to have $(x,y,u,v)$ in the unitary tangent bundle of $M_1 \times M_2$ we take $(x,y,\alpha u,\beta v)$, and
$$g_{(x,y)}((\alpha u, \beta v),(\alpha u, \beta v)) = \alpha^2 g^1_x(u,u) + \beta^2 g^2_y(v,v) = \alpha^2 + \beta^2 = 1.$$
Then $$\gamma_{(x,y,\alpha u,\beta v)}(t) = (\gamma_{(x,\alpha u)}(t),\gamma_{(y,\beta v)}(t)), \gamma'_{(x,y,\alpha u,\beta v)}(t) = (\alpha \gamma'_{(x,u)}(t),\beta \gamma'_{(y,v)}(t)).$$
\noindent Let $E_i$, $i=2,\ldots,dim(M_1)$, be an orthogonal frame of parallel vector fields along the geodesic $\gamma_{(x,u)}$. Let $F_j$, $j=2,\ldots,dim(M_2)$, be an orthogonal frame of parallel vector fields along the geodesic $\gamma_{(y,v)}$. 

Notice that along the geodesic $\gamma_{(x,y,\alpha u,\beta v)}$, since its components are $\gamma_{(x,\alpha u)}$ and $\gamma_{(y,\beta v)}$, the following holds:
$$g^1_{\gamma_{(x,\alpha u)}(t)}(\gamma'_{(x,\alpha u)}(t),\gamma'_{(x,\alpha u)}(t)) = \alpha^2, g^2_{\gamma_{(y,\beta v)}(t)}(\gamma'_{(y,\beta v)}(t),\gamma'_{(y,\beta v)}(t)) = \beta^2,$$ 
\noindent so the proportion $(\alpha,\beta)$ is preserved along the geodesic.

So $\{(\alpha \gamma'_{(x,u)}(t),\beta \gamma'_{(y,v)}(t)),(\beta \gamma'_{(x,u)}(t),- \alpha \gamma'_{(y,v)}(t)),(E_i(t),0),(0,F_j(t))\}_{i,j}$ is an ortho\-normal frame of parallel vector fields along the geodesic $\gamma_{(x,y,\alpha u,\beta v)}(t)$.

The fact that the second fundamental form of the submanifolds $\{p\} \times M_2$ and $M_1 \times \{q\}$ is zero, together with Gauss and Codazzi equations, imply that:
$$R((u_1,0),(u_2,0),(u_3,0),(u_4,0)) = R^1(u_1,u_2,u_3,u_4),$$
$$R((0,v_1),(0,v_2),(0,v_3),(0,v_4)) = R^2(v_1,v_2,v_3,v_4),$$
$$R((u_1,0),(u_2,0),(u_3,0),(0,v_1)) = 0,$$
$$R((0,v_1),(0,v_2),(0,v_3),(u_1,0)) = 0.$$
\noindent Also the fact that the curvature is zero for planes generated by one vector tangent to $M_1$ and another tangent to $M_2$ implies:
$$R((u_1,0),(0,v_1),(u_2,0),(0,v_2)) = 0.$$
\noindent All these equations imply that along the geodesic $\gamma_{(x,y,\alpha u,\beta v)}(t)$:
$$R(\gamma'_{(x,y,\alpha u,\beta v)},(E_i,0),\gamma'_{(x,y,\alpha u,\beta v)},(E_k,0)) = \alpha^2 R^1(\gamma'_{(x,u)},E_i,\gamma'_{(x,u)},E_k),$$
$$R(\gamma'_{(x,y,\alpha u,\beta v)},(0,F_j),\gamma'_{(x,y,\alpha u,\beta v)},(0,F_l)) = \beta^2 R^2(\gamma'_{(y,v)},F_j,\gamma'_{(y,v)},F_l),$$
$$R(\gamma'_{(x,y,\alpha u,\beta v)},(E_i,0),\gamma'_{(x,y,\alpha u,\beta v)},(0,F_j)) = 0.$$
\noindent Now, we are going to write a system of Jacobi fields. If we have a Jacobi field $\zeta(t) = \sum_{i=2}^n f_i(t) U_i(t)$, then $\zeta''(t) = \sum_{i=2}^n f''_i(t) U_i(t)$ and 
$$0 = \sum_{j=2}^n (f''_j + \sum_{i=2}^n f_i R(\gamma',U_i,\gamma',U_j)) U_j.$$
\noindent So, the Jacobi equation can be written as:
$$\begin{bmatrix} f \\ f' \end{bmatrix}' = \begin{bmatrix} 0 & I \\ - K & 0 \end{bmatrix} \begin{bmatrix} f \\ f' \end{bmatrix}$$
\noindent where $K_{ij} = R(\gamma',U_i,\gamma',U_j)$.

In the case of the product metric we have:
$$\begin{bmatrix} f \\ f' \end{bmatrix}' = \begin{bmatrix} 0 & 0 & I & 0 \\ 0 & 0 & 0 & I \\ - \alpha^2 K^1 & 0 & 0 & 0 \\ 0 & - \beta^2 K^2 & 0 & 0 \end{bmatrix} \begin{bmatrix} f \\ f' \end{bmatrix}.$$
\noindent With a change in the order of the basis of parallel vector fields we have:
$$F' = \begin{bmatrix} 0 & I & 0 & 0 \\ - \alpha^2 K^1 & 0 & 0 & 0 \\ 0 & 0 & 0 & I \\ 0 & 0 & - \beta^2 K^2 & 0 \end{bmatrix} F.$$
\noindent So the systems decouples and the solutions are given immediately by the solutions for $M_1$ and $M_2$.

Now suppose the geodesic flow of the product metric is partially hyperbolic with splitting $E^s \oplus E^c \oplus E^u$, $dim E^s = p$, $dim E^u = q$. So the geodesic flow of each metric $g^1$ and $g^2$ is partially hyperbolic and each geodesic flow inherits a partially hyperbolic splitting:
$$E_1^s \oplus E_1^c \oplus E_1^u,$$
\noindent along geodesics in $M_1 \times \{y\}$ ($\beta = 0$), such that $E_1^s \oplus E_1^u \subset T_xM_1 \oplus \{0\} \subset T_xM_1 \oplus T_yM_2$, and
$$E_2^s \oplus E_2^c \oplus E_2^u,$$
\noindent along geodesics in $\{x\} \times M_2$ ($\alpha = 0$), such that $E_2^s \oplus E_2^u \subset \{0\} \oplus T_yM_2 \subset T_xM_1 \oplus T_yM_2$.

For geodesics of the product metric which have $\alpha \neq 0 \neq \beta$, we get a splitting into five invariant subbundles $E_1^s \oplus E_2^s \oplus E^c \oplus E_1^u \oplus E^u_2$, without the domination, since $\alpha$ and $\beta$ multiply the Lyapunov exponents of each subbundle. Since we already have an splitting, $E^s$ and $E^u$ are necessarily one of a combination of subbundles of $E^s_1$ and $E^s_2$, $E^u_1$ and $E^u_2$, respectively:
$$E^s \in \{ E \oplus F : E \subset E^s_1, F \subset E^s_2, dim E + dim F = p\},$$
$$E^u \in \{ E \oplus F : E \subset E^u_1, F \subset E^u_2, dim E + dim F = q\}.$$
\noindent So there is no way to go from the case $\alpha = 0$ to $\beta = 0$ without breaking the continuity of the splitting, because one cannot go from the case dim $E = 0$ for $\beta = 0$, to dim $F = 0$ for $\alpha = 0$ in a continuous way. 

\end{proof}

\section{Anosov geodesic flow with many invariant subbundles}
\label{s.orig}

In this section, we introduce the metric which we are going to deform to produce the example of a partially hyperbolic and non-Anosov geodesic flow.

The candidate for the deformation is a compact locally symmetric space which is a quotient of the symmetric space of nonconstant negative curvature $M := G / K$ by a cocompact lattice  $\Gamma$, where $G$ is a Lie group that acts transitively on $M$ and $K$ is a compact subgroup of $G$ \cite{Bo}. 

Cartan classified the symmetric spaces of negative curvature (see \cite{He}, \cite{Hel}). They are:

\begin{itemize}
\item[i.] the hyperbolic space $\mathbb{R}H^n$ of constant curvature $-c^2$, which is the canonical space form of negative constant curvature;
\item[ii.] the hyperbolic space $\mathbb{C}H^n$ of curvature $-4c^2 \leq K \leq - c^2$, which is the canonical K\"ahler hyperbolic space of constant negative holomorphic curvature $- 4 c^2$ \cite{G};
\item[iii.] the hyperbolic space $\mathbb{H}H^n$ of curvature $-4c^2 \leq K \leq - c^2$, which is the canonical quaternionic K\"ahler symmetric space of negative curvature \cite{Be}, \cite{Wo};
\item[iv.] the hyperbolic space $Ca H^2$ of curvature $-4c^2 \leq K \leq - c^2$, which is the canonical hyperbolic symmetric space of the octonions of constant negative curvature. 
\end{itemize}

Their geodesic flows are all Anosov, but the geodesic flow of the first one does not have more than the two invariant subbundles, the stable and the unstable, which cannot be decomposed in other subbundles. The others have more invariant subbundles, as in the first item of the strategy written in section \ref{s.intro}. So, the metrics which are the candidates to produce a partially hyperbolic geodesic flow which is not Anosov are the metrics in items [ii.], [iii.] and [iv.]. Through the article we are going to consider $c = \frac{1}{2}$.

For these type of metrics we need the following properties to hold:

\begin{itemize}
\item[i.] For all $v \in T_xM$, the subspace $\{ w \in T_xM : K(v,w) = -1 \}$ is parallel along $\gamma_v$ in the sense  that the derivative of the projection to this subspace of $T_xM$ along geodesics is zero; 
\item[ii.] For closed geodesics $\gamma: [0,T] \to M$, $(\gamma(0),\gamma'(0)) = (\gamma(T),\gamma'(T))$, the parallel translation from $\gamma(0)$ to $\gamma(T)$ along $\gamma$ of these subspaces $\{ w \in T_xM : K(v,w) = -1 \}$ and $\{ w \in T_xM : K(v,w) = -\frac{1}{4} \}$, where $v = \gamma'(0)$, preserves orientation.
\end{itemize}

The examples that satisfy the properties above are:

\begin{itemize}
\item[i.] compact K\"ahler manifolds of negative holomorphic curvature $-1$ (see \cite{G}), \item[ii.] compact locally symmetric quaternionic K\"ahler manifolds of negative curvature (see \cite{Be}). 
\end{itemize}

In the next subsection we describe a series of properties of the hyperbolic subbundles of the Anosov geodesic flows for the two cases listed above. More precisely, in subsection \ref{s.sub} we describe the strong and weak stable and unstable subbundles, in subsection \ref{s. cone var for anosov} we study the cone variation (as defined in subsection \ref{s. cone variation def}) for the hyperbolic subbundles and in subsection \ref{s.orient} we study the orientability of those subbundles.

\subsection{Subspaces of $S(UM)$ and $UM$}
\label{s.sub}

Since the candidate has nonconstant negative curvature, then its sectional curvature, up to multiplication of the metric by a constant, has planes of sectional curvature $-1$ and planes of sectional curvature $-\frac{1}{4}$. Actually, every vector $v \in TM$ is in a plane with curvature $-1$ and in another with curvature $-\frac{1}{4}$. 

We define 
\begin{eqnarray}
\label{d.AB} 
A(x,v) & := & \{ w \in T_xM : K(v,w) = -1 \}, \\ B(x,v) & := & \{ w \in T_xM : K(v,w) = -\frac{1}{4} \}.
\end{eqnarray} 

If we restrict the derivative of the geodesic flow to the subbundle $S(UM) = ker \alpha \to UM$, where $\alpha$ is the contact form on $UM$, then $S(x,v) = \widehat{H}(x,v) \oplus \widehat{V}(x,v)$, and $\widehat{H}(x,v)$ and $\widehat{V}(x,v)$ are identified with $\{v\}^{\bot} = A(x,v) \oplus B(x,v) \subset T_xM$, $(x,v) \in UM$. The subbundles $A$ and $B$ are invariant by parallel translation along  geodesics. 

\begin{lemma}
The geodesic flow of the symmetric spaces of nonconstant negative curvature induces a hyperbolic splitting of the contact structure defined on $UM$: $S(UM)/<X> = E^{ss} \oplus E^s \oplus E^u \oplus E^{uu}$.
\end{lemma}

\begin{proof}
We can define the invariant subbundles $P^u_C(v), P^s_C(v) \subset T_{(x,v)}UM$, $C = A, B$ such that
$$P^u_C(v) = \{ (w,\alpha_C w) \in S(x,v) : w \in C(x,v) \},$$ $$P^s_C(v) = \{ (w,-\alpha_C w) \in S(x,v) : w \in C(x,v) \}.$$  
\noindent where $\alpha_A = 1$ and $\alpha_B = \frac{1}{2}$.

That invariant subbundles are exactly the subbundles of the decomposition in the first item of the strategy stated in the introduction: 
$$E^{uu}(x,v) = P^u_A(x,v),$$ 
$$E^{ss}(x,v) = P^s_A(x,v),$$ 
$$E^s(x,v) = P^s_B(x,v),$$ 
$$E^u(x,v) = P^u_B(x,v).$$
Above subbundles are invariants  and the splitting is dominated: Jacobi fields in $E^{uu}$ and $E^{ss}$ contract for the past and the future, respectively, at rate $e^{-t}$ and Jacobi fields in $E^u$ and $E^s$ contract for the past and the future, respectively, at rate $e^{-t/2}$.
\end{proof}

\subsubsection{Angle cone variation for the Anosov flow with many subbundles}
\label{s. cone var for anosov}

Let us calculate the proper invariance of the cones in the case of the geodesic flow of the compact locally symmetric Riemannian manifold of nonconstant negative sectional curvature. 

We use the following family of trajectories for the system:
$$q(t,u) = \pi \circ \phi_t(z(u)),$$
\noindent $|u| < \epsilon$. We consider the geodesic $v(t) := \phi_t(z(0)) = \gamma'_{z(0)}(t)$.  

\noindent The Jacobi system along the geodesic $\gamma_{z(0)}$ is given by $(\xi(t),\eta(t))$, where
$$\xi(t) = \frac{dq}{du}(t,0) ,\eta(t) = \frac{Dv}{du}(t,0) = \frac{D}{du} \frac{dq}{dt}(t,0).$$
\noindent So the following equations hold:
\begin{equation}
\label{eq:jacobi}\frac{D\xi}{dt} = \eta, \frac{D\eta}{dt} = - R(v,\xi)v.
\end{equation}
\noindent The quantity (\ref{eq.cone0}), which in this case is
$$\Theta_A^u(\xi,\eta) := \frac{g(Pr_A(\xi+\eta),Pr_A(\xi+\eta))}{g(\xi,\xi) + g(\eta,\eta)},$$
\noindent indicates twice the square of the cosine of the angle between the vector $(\xi,\eta) \in T_{(x,v)}UM$ and its projection to $P^u_A(x,v)$. The cone in this case is $$C(v,P^u_A(x,v),c) = \{(\xi,\eta) \in T_{(x,v)}UM : \Theta_A^u(\xi,\eta) = \frac{\widehat{g}(Pr_{P^u_A(v)} (\xi,\eta), Pr_{P^u_A(v)} (\xi,\eta))}{\widehat{g}((\xi,\eta),(\xi,\eta))} \geq c\},$$ where $\widehat{g}$ is the Sasaki metric and $c = 2\cos^2 \delta$, where $\delta$ is the angle between the vectors in the boundary of the cone and the subspace $P^u_A(v)$. So, as explained in subsection \ref{s. cone variation def},  to prove that the cone fields are properly invariant  is equivalent  to prove that the cosine of this angle increases under the action of the derivative of the geodesic flow, for vector in the boundary of the cone fields, 
$$\partial C(v,P^u_A(x,v),c) = \{ (\xi,\eta) \in T_{(x,v)}UM : \Theta_A^u(\xi,\eta) = c \in (1,2) \}.$$
Remember that for any Riemannian manifold $(M,g)$ and any $u(t),v(t) \in T_{\gamma(t)}M$ vector fields along a geodesic $\gamma: \mathbb{R} \to M$ on $M$ 
$$\frac{d}{dt} g(u,v) = g\left(\frac{Du}{dt},v\right) + g\left(u,\frac{Dv}{dt}\right).$$
If $(M,g)$ is locally symmetric, and if $\xi(t) \in T_{\gamma(t)}M$ is a vector field along a geodesic $\gamma(t)$ on $M$, then
$$\frac{D}{dt} Pr_A \xi = Pr_A \frac{D}{dt} \xi,$$
\noindent i.e. the subspace $A(\gamma'(t))$ is parallel along the geodesic $\gamma: \mathbb{R} \to M$ of $(M,g)$.
Let us call, to simplify the equations, $$\xi_A := Pr_A \xi, \xi_B := Pr_B \xi,$$ $$\xi'_A = Pr_A \frac{D}{dt} \xi = (\xi_A)', \xi'_B = Pr_B \frac{D}{dt} \xi = (\xi_B)',$$ 
\noindent and remember that $\xi = \xi_A + \xi_B$. Then, for 
$$\frac{g(\xi_A+\eta_A,\xi_A+\eta_A)}{g(\xi,\xi) + g(\eta,\eta)} = \Theta_A^u(\xi,\eta) = c \in (1,2),$$
\noindent its derivative along a geodesic $\gamma: \mathbb{R} \to M$ is: 
\begin{eqnarray*}
\frac{d}{dt} \Theta_A^u(\xi,\eta) & = & 2 \frac{g(\xi_A+\eta_A,\xi'_A + \eta'_A + \xi_{A'} + \eta_{A'})}{g(\xi,\xi) + g(\eta,\eta)} \\ & - & 2 \frac{g(\xi_A+\eta_A,\xi_A+\eta_A)}{(g(\xi,\xi) + g(\eta,\eta))^2} (g(\xi,\xi') + g(\eta,\eta')). 
\end{eqnarray*}
\noindent Since the subspaces $A(\gamma'(t))$ and $B(\gamma'(t))$ are parallel along the geodesic $\gamma: \mathbb{R} \to M$ and equation \ref{eq:jacobi} holds, then
\begin{eqnarray*}
\xi_{A'} = 0 & , & \eta_{A'} = 0 \\
\xi' = \eta & , & \eta' = - R(v,\xi)v \\
\xi'_A = \eta_A & , & \eta'_A = - R(v,\xi)v_A,  
\end{eqnarray*}
\noindent imply
\begin{eqnarray*} \frac{d}{dt} \Theta_A^u(\xi,\eta) & = & 2 \frac{g(\xi_A+\eta_A,\eta_A - (R(v,\xi)v)_A)}{g(\xi,\xi) + g(\eta,\eta)} \\ & - & 2 \frac{g(\xi_A+\eta_A,\xi_A+\eta_A)}{(g(\xi,\xi) + g(\eta,\eta))^2} (g(\xi,\eta) - R(v,\xi,v,\eta)). \end{eqnarray*}
\noindent But for the locally symmetric metric of negative curvature, the curvature is:
$$R(v,\xi)v = - \frac{1}{4} \xi_B - \xi_A.$$
\noindent So, we have:
\begin{eqnarray*}\frac{d}{dt} \Theta_A^u(\xi,\eta) & = & 2 \frac{g(\xi_A + \eta_A,\xi_A+\eta_A)}{(g(\xi,\xi) + g(\eta,\eta))^2} (g(\xi,\xi) + g(\eta,\eta) \\ & - & g(\xi,\eta) - g(\xi_A,\eta_A) - \frac{1}{4} g(\xi_B,\eta_B )).
\end{eqnarray*}

We need to show that it is positive at least in the boundary of the cone $C(v,P^u_A(v),c)$. In fact it will be positive for any initial $c \in (1,2)$:

\begin{eqnarray*} \frac{d}{dt} \Theta_A^u(\xi,\eta) & = & \frac{2g(\xi_A + \eta_A,\xi_A+\eta_A)}{(g(\xi,\xi) + g(\eta,\eta))^2} (g(\xi_A,\xi_A) + g(\xi_B,\xi_B) + g(\eta_A,\eta_A) \\ & & + g(\eta_B,\eta_B) - g(\xi_A,\eta_A) - g(\xi_B,\eta_B) - g(\xi_A,\eta_A) - \frac{1}{4} g(\xi_B,\eta_B)) \\ & = & \frac{2g(\xi_A + \eta_A,\xi_A+\eta_A)}{(g(\xi,\xi) + g(\eta,\eta))^2} (g(\xi_A,\xi_A) - 2 g(\xi_A,\eta_A) \\ & & + g(\eta_A,\eta_A) + g(\xi_B,\xi_B) - \frac{5}{4} g(\xi_B,\eta_B) + g(\eta_B,\eta_B)) \\ & = & \frac{2g(\xi_A + \eta_A,\xi_A+\eta_A)}{(g(\xi,\xi) + g(\eta,\eta))^2}(g(\xi_A - \eta_A,\xi_A - \eta_A) \\ & & + g(\xi_B - \frac{5}{8} \eta_B,\xi_B - \frac{5}{8} \eta_B) + \frac{39}{64}g(\eta_B,\eta_B)). \end{eqnarray*}
\noindent Since the derivative is the same if $(\xi,\eta)$ is multiplied by a scalar, we consider $(\xi,\eta)$ such that $g(\xi,\xi) + g(\eta,\eta) = 1$, and such that they are in the boundary of the cones  of size $c$. This is a compact set and the derivative for this values of $(\xi,\eta)$ is far away from zero: $$\frac{d}{dt} \Theta_A^u(\xi,\eta) = g(\xi_A - \eta_A,\xi_A - \eta_A) + g(\xi_B - \frac{5}{8} \eta_B,\xi_B - \frac{5}{8} \eta_B) + \frac{39}{64}g(\eta_B,\eta_B).$$ This means that the cones are properly invariant under the action of the derivative of the geodesic flow.

To get the exponential growth, we need lemma \ref{l.contreras}, but in this case we are able to calculate it explicitly:
\begin{eqnarray*} \frac{d}{dt} g(\xi_A+\eta_A,\xi_A+\eta_A) & = & 2 g(\xi_A+\eta_A,\eta_A - (R(v,\xi)v)_A)
\\ & = & 2 g(\xi_A+\eta_A,\xi_A+\eta_A). \end{eqnarray*}
\noindent This implies that the vectors inside the cone grow at the rate of $e^t$.

\subsubsection{Orientability of $A$ and $B$}
\label{s.orient}

In following section, we use  normal coordinates along closed geodesics. To define them properly in a suitable way for our needs, we  need that the hyperbolic subbundles are  orientable. We discuss this issue in the present section.  

Recall that a K\"ahler manifold is a triple $(M,J,\omega)$, such that $J: TM \to TM$ is a integrable complex map with $J^2 = -Id_{TM}$, and $\omega$ is a $J$-compatible symplectic form. In the case of negative holomorphic curvature $-1$, $A(x,v) = \mathbb{R} \cdot Jv$ and $B(x,v)$ has a basis of the form $(e_1,Je_1,\ldots,e_k,Je_k)$. If $\gamma$ is a closed geodesic then the parallel transport along $\gamma$ sends $Jv$ to $Jv$ and sends $(e_1,Je_1,\ldots,e_k,Je_k)$ to $(\tilde{e}_1,J\tilde{e}_1,\ldots,\tilde{e}_k,J\tilde{e}_k)$, which have the same orientation.

In the K\"ahler quaternionic case, instead of one map $J$, there are three maps $I$, $J$, $K$, such that $I^2 = J^2 = K^2 =  -Id_{TM}$, $IJ = - JI$, $K = IJ$ \cite{Be}, \cite{Wo}. In this case, $A(x,v)$ has as its basis $(Iv,Jv,Kv)$. The three maps are not parallel, but the orthogonal projection to $A$ is parallel. Also, $Q(v) = Iv \wedge Jv \wedge Kv$ is parallel, so along closed geodesics the orientation of $A(x,v)$ is preserved \cite{Gr}. For the same reason, $Q$ being parallel, $B(x,v)$ has its orientation preserved along closed geodesics.

\section{The partially hyperbolic non-Anosov example}
\label{s.nonanosov}

In the first subsection we give a more detailed strategy for the deformation of the metric introduced in the previous section. 

In the second subsection we give some definitions and we introduce the deformation of the original metric whose geodesic flow is partially hyperbolic and non-Anosov.

In the subsections \ref{s.para1}, \ref{s.para2}, \ref{s.transv} we show that the new geodesic flow preserves a strong stable and a strong unstable cone fields. We first show that along the closed geodesic $\gamma$ the strong stable and strong unstable cones are properly invariant under the action of the derivative of the deformed geodesic flow. Then, we show that for geodesics which are close to $\gamma'=(v_0,0,0,\ldots,0)$ the strong stable and strong unstable cones are properly invariant too (subsections \ref{s.para1} and \ref{s.para2}). Then we show that for geodesics that cross the neighborhood of the deformation of the compact locally symmetric metric the strong stable and strong unstable cones are not properly invariant, but we manage to control the lack of this property in such a way that, after crossing the neighborhood, and inside the region where the metric remains the same, proper invariance is obtained (subsection \ref{s.transv}). Then we prove that there is expansion for the vectors in the strong unstable cones, and contraction for the vectors in the strong stable cones (subsection \ref{s.expo}).

In subsection \ref{s.conc} we state the main theorem and some of its corollaries. 

\begin{remark}
\label{r.reverse}
We only need to show the strong unstable cone is properly invariant, because this guarantees that we have one unstable subbundle $E^u$ invariant under the flow. For the same reasons there is a properly invariant unstable subbundle for the inverse of the flow, which is the stable subbundle, since geodesic flows are reversible flows.
\end{remark}

\subsection{The strategy to construct the example}
\label{s.strateg}

First  we add more details to the strategy of the proof of theorem A  described in the introduction:

\begin{itemize}
\item[1.] We chose a metric whose geodesic flow is Anosov and whose hyperbolic invariant splitting is of the form $T(UM) = E^{ss} \oplus E^s \oplus \left\langle X \right\rangle \oplus E^u \oplus E^{uu}$ (recall section \ref{s.orig});
\item[2.] We take a closed geodesic $\gamma$ without self-intersections (section \ref{s.prop});
\item[3.] We change the metric in a tubular neighborhood of $\gamma$ in $M$, such that along $\gamma$ the strong subbundles ($E^{ss}$ and $E^{uu}$) remain invariant and the weak subbundles disappear, becoming a central subbundle with no hyperbolic behavior (section \ref{s.geode}):
\item[3.1.] to obtain the non-hyperbolicity we change the metric in such a way that the directions of small curvature become directions of zero curvature (section \ref{s.geode} and \ref{s.geod nh});
\item[3.2.] to obtain that the strong subbundles remain the same along $\gamma$ we deform the metric along  it in such a way that the directions of larger curvature ($E^{ss}$ and $E^{uu}$) remain (sections \ref{s.geode} and \ref{s.phgamma});
\item[4.] Observe that outside the tubular neighborhood of $\gamma$, the dynamics remains hyperbolic;
\item[5.] We show that for the geodesics that intersect the tubular neighborhood the cones associated to the extremal subbundles ($E^{ss}$ and $E^{uu}$)  are preserved (sections \ref{s.para1}, \ref{s.para2}, \ref{s.transv}. \ref{s.phcone}):
\item[5.1.]  for  geodesics close to  $\gamma_0$ ('parallel' region), we verify that the cones associated with the extremal subbundles are preserved (sections \ref{s.para1}, \ref{s.para2});
\item[5.2.] for geodesics 'transversal' to  $\gamma_0$ ('transversal' region) we control the angle cone variation for the cones associated to the extremal subbundles with its own axis under the action of the derivative of the geodesic flow (section \ref{s.transv});
\item[5.3.] then, we prove that for any geodesic the time spent in the 'transversal' region is  small as we need in comparison to the time spent outside it (section \ref{s.transv});
\item[6.] We finish proving that for vectors in the unstable cones there is expansion, and for vectors inside the stable cones there is contraction, under the action of the derivative of the new geodesic flow (section \ref{s.expo}).
\end{itemize}

\subsection{The new metric $g^*$ and its properties}
\label{s.prop}

Let us call $(M^n,g)$ a compact locally symmetric space of nonconstant negative curvature of dimension $n$, introduced in section \ref{s.orig}.

Let us fix a closed prime geodesic $\gamma: [0,T] \to M^{n}$, with $\gamma(0) = \gamma(T)$ and $\gamma'(0) = \gamma'(T)$, without self-intersections. This is the closed geodesic which we use to construct the tubular neighborhood where we change the metric $g$. There is always a geodesic with these properties in a compact Riemannian manifold \cite{Kl}.

\begin{definition}
Let us define a  tubular neighborhood of the geodesic $\gamma$, constructed as follows:

We introduce normal coordinates along this geodesic. Take an orthonormal basis of vector fields $\{ e_0(t):=\gamma'(t),e_1(t),\ldots,e_{n-1}(t)\}$ in $T_{\gamma(t)}M$, such that $\{e_1(t),\ldots,e_r(t)\}$ is a basis for $A(\gamma(t),\gamma'(t))$, and $\{e_{r+1}(t),\ldots,e_{n-1}(t)\}$ is a basis for $B(\gamma(t),\gamma'(t))$. This is possible because the parallel transport preserves orientation and M is orientable. $\Psi: [0,T] \times (- \epsilon_0,\epsilon_0)^{2n-1} \to M : (t,x) \to \exp_{\gamma(t)} ( x_1 e_1(t) + x_2 e_2(t) + \ldots + x_{n-1} e_{n-1}(t))$ with $\epsilon_0$ less than the injectivity radius, so $\Psi|_U$ is a diffeomorphism, with $U = [0,T] \times (- \epsilon_0,\epsilon_0)^{n-1}$. We define $U(\epsilon) := [0,T] \times (- \epsilon,\epsilon)^{n-1}$. Now, the tubular neighborhood is noted and defined as  $  {\mathcal B}(\gamma,\epsilon)  = \Psi(U(\epsilon)).$
\end{definition}

\begin{definition}
The set of vectors $\{ (x,v) \in UM : x \in  {\mathcal B}(\gamma,\epsilon) , |v_i| < \theta, i=1,\ldots,n-1 \}$ is called the set of $\theta$-parallel vectors to $\gamma$, the set $\{ (x,v) \in UM : x \in  {\mathcal B}(\gamma,\epsilon) , |v_i| \geq \theta, \textrm{ for some } i=1,\ldots,n-1 \}$ is called the set of $\theta$-transversal vectors to $\gamma$. If $(x,v) \in UM$ belongs to the set of $\theta$-parallel vectors for all $\theta$, then we call it a parallel vector to $\gamma$. Notice that $\{ (x,v) \textrm{ is } \theta \textrm{-parallel to } \gamma \} \cup \{ (x,v) \textrm{ is } \theta \textrm{-transversal to } \gamma \} =  {\mathcal B}(\gamma,\epsilon) $.
\end{definition}

Let $g_{ij}(t,x)$ denote the components of the metric in this tubular neighborhood of $\gamma$ where $\Psi$ is defined. We define a new Riemannian metric $g^*$ as:
$$g^*_{00}(t,x) := g_{00}(t,x) + \alpha(t,x),$$
$$\alpha(t,x) := \sum_{i,j=1}^{n-1} \Phi_{ij}(t,x) x_i x_j,$$
$$g^*_{ij}(t,x) := g_{ij}(t,x), (i,j) \neq (0,0),$$
\noindent with $\Phi_{ij}: [0,T] \times (- \epsilon_0,\epsilon_0)^{n-1} \to \mathbb{R}$, where each $\Phi_{ij}$ is a bump function. This kind of deformation allows us to change the curvature (change the second derivative), as $\gamma$ and the parallel transport along $\gamma$ (the metric up to its first derivative) remain the same. This becomes clear if we look to the formulas of the metric, the parallel transport and the curvature with respect to a coordinate system.

For this new metric $g^*$, the coordinates along $\gamma$ are:
$$g^{*ij}(t,0) = g^{ij}(t,0), 0 \leq i,j \leq n-1,$$
$$g^*_{ij}(t,0) = g_{ij}(t,0), 0 \leq i,j \leq n-1,$$
$$\partial_k g^{*ij}(t,0) = \partial_k g^{ij}(t,0),0 \leq i,j,k \leq n-1,$$
$$\partial_k g^*_{ij}(t,0) = \partial_k g_{ij}(t,0), 0 \leq i,j,k \leq n-1.$$
\noindent These equalities imply that the closed geodesic $\gamma$ still is a closed geodesic for $g^*$. We are going to use the following deformation:
$$\alpha(t,x) =  \sum_{k=r+1}^{n-1}x^2_{k} \Phi_k(t,x),$$
The first property we need for the function $\alpha: U \to \mathbb{R}$ is that $\Phi_k(t,0) = - \frac{1}{4}$. The $\Phi_k$ are going to be products of bump functions define on a tubular neighborhood of $\gamma$ of radius $\epsilon < \epsilon_0$. We need to change $\epsilon$ along the proof, so we can say that this functions $\Phi_k$ are going to be $\epsilon$-parameter families of functions. For some $\epsilon$ small enough the new metric $g^*$ is going to be partially hyperbolic. 

\begin{remark}
Notice that the functions $\alpha$ and $\Phi_k$ are defined in a tubular neighborhood of the closed geodesic $\gamma$. Since $t$ is the coordinate of the closed geodesic $\gamma$ and we want to preserve it after perturbations, we do not use bump functions with $t$ as a parameter. So we define $\alpha$ as a linear combination of products of bump functions $\Phi_k$ which do not depend on $t$. 
\end{remark}

Let us construct some functions that will help us define the $\alpha$ function we need.

\begin{definition} 
Let us define for a positive real number $h$ and a non-negative real number $\tau$ the function $\varphi_{h,\tau}: \mathbb{R}_{\geq 0} \to \mathbb{R}$, continuous and piecewise-$C^1$ with support in $[0,1]$:

\begin{itemize}
\item[.] $\varphi_{h,\tau}(0) = \varphi_{h,\tau} (1) = \varphi_{h,\tau}(\frac{1}{2}) = 0$,
\item[.] $\varphi_{h,\tau}'(x) = \frac{h}{\tau}$, if $x \in (0,\tau) \cup (1 - \tau,1)$,
\item[.] $\varphi_{h,\tau}'(x) = - \frac{h}{\tau}$, if $x \in (\frac{1}{2} - \tau, \frac{1}{2} + \tau)$,
\item[.] $\varphi_{h,\tau}(x) = h$ for $x \in (\tau,\frac{1}{2} - \tau)$, $\varphi_{h,\tau}(x) = - h_{\tau}$, for $x \in (\frac{1}{2} - \tau,1 - \tau)$,
\item[.] if $\tau = 0$, $\varphi_{h,0} := h \chi_{[-\frac{1}{2},\frac{1}{2}]} - h \chi_{[-1,-\frac{1}{2})} - h \chi_{(\frac{1}{2},1]}$.
\end{itemize}
\end{definition}

\begin{definition}
Let us define $\phi_{h,\tau}: \mathbb{R}_{\geq 0} \to \mathbb{R}$ such that $\phi_{h,\tau}(x) := \int_x^1 \int_0^s \varphi_{h,\tau}(t) dt ds$.
\end{definition}

\begin{lemma}
For each $\tau \geq 0$ there is a $h_{\tau}$ such that $\phi_{h_{\tau},\tau}(0) = - \frac{1}{2}$. Moreover, the function $\mathcal{H}: \mathbb{R}_{\geq 0} \to \mathbb{R}$ such that $\mathcal{H}(\tau) = h_{\tau}$ is a $C^1$-function.
\end{lemma}
\begin{proof}
It is easy to see that $\phi_{h,\tau}(0) = - \frac{h}{4}(1 - 2 \tau)$, so $\partial_{\tau} \phi_{h,\tau}(0) = \frac{h}{2} \neq 0$, so by the implicit function theorem there is only one $h_{\tau}$ such that $\phi_{h_{\tau},\tau}(0) = - \frac{1}{2}$ and $\mathcal{H}(\tau) = h_{\tau}$ is a $C^1$-function.
\end{proof}

\begin{definition}
Let us define for a non-negative real number $\tau$ the function $\phi_{\tau}: \mathbb{R}_{\geq 0} \to \mathbb{R}$ such that $\phi_{\tau} := \phi_{h_{\tau},\tau}$.
\end{definition}

\begin{lemma}
Let $F_{\tau}: \mathbb{R}_{\geq 0} \to \mathbb{R}$ be the function $F_{\tau}(x) := x^2 \phi_{\tau}''(x) + 4x \phi_{\tau}'(x) + 2 \phi_{\tau}(x)$. Then For every $\delta > 0$ there is a small enough positive real number $\beta$ such that if $\tau < \beta$ then $F_{\tau}(x) \in [-2(1+\delta)F_{\tau}(0),2(1+\delta)F_{\tau}(0)]$.
\end{lemma}

\begin{proof}

We use the following facts that hold for $\tau$ small enough:
\begin{itemize}
\item[.] the quadratic term of $F_{\tau}$ is the only one that does not varies continuously as $\tau$ varies. The other two do vary continuously because $\phi_{\tau}$ is $C^1$-close to $\phi_0$,
\item[.] $h_{\tau}$ depends $C^1$ on $\tau$,
\item[.] $F_{\tau}(0) = - 1$ and so does not depend on $\tau$.
\end{itemize} 

For $\tau = 0$, we have that $F_0(x) = (- \frac{1}{2} + 6x^2)h_0$ for $x \in (0,\frac{1}{2})$ and $F_0(x) = (- 1 + 6x + 6x^2)h_0$ for $x \in (\frac{1}{2},1)$. Then it is simple to see that $F_0(0) = - \frac{h_0}{2}$ and $F_0(x) \in [- h_0, h_0]$. Then, for $\tau = 0$ we have that $F_0(x) \in [-2F_0(0),2F_0(0)]$.

So for $\delta > 0$ there is a $\beta > 0$ such that $h_{\tau} \in ((1-\delta) h_0,(1+\delta) h_0)$. We know by definition of $\phi_{\tau}$ that $\phi_{\tau}''(x) \in [-h_{\tau},h_{\tau}]$, which implies $x^2 \phi_{\tau}''(x) \in [-(1+\delta)h_0,(1+\delta) h_0]$. 

We suppose also that $\beta$ is small enough so that $F_{\tau}$ minus its quadratic part is $\delta$ close to $F_0$ minus its quadratic part. Then $F_{\tau}(x) \in [-2(1+\delta)F_0(0),2(1+\delta)F_0(0)]$ but $F_0(0) = F_{\tau}(0)$ so the statement of the lemma is proved.
 
\end{proof}

\begin{remark}
Let us define the function $\phi^{\lambda}_{\tau}: \mathbb{R} \to \mathbb{R}$ as $\phi^{\lambda}_{\tau}(x) := \phi_{\tau}(\frac{x}{\lambda})$. Our bump functions $\phi^{\lambda}_{\tau}$ have support in an interval of length $2 \lambda$, so let us notice that if the lemma holds for $\phi_{\tau}$ with support in $[0,1]$, then it holds for $\phi^{\lambda}_{\tau}$ for any $\lambda$. It also holds if $\phi_{\tau}$ is multiplied by a constant. And it also holds if $\phi$ is a $C^{\infty}$ function $C^2$ close to $\phi_{\tau}$.
\end{remark}

\begin{definition}
Let the function $\phi_{k,j}: \mathbb{R} \to \mathbb{R}$ be such that $$\phi_{k,j}(x) = \left\{ \begin{aligned} \phi^{\epsilon}(x) & \text{ if } k \neq j, \\ \phi^{\epsilon^2}(x) & \text{ if } k=j,  \end{aligned} \right.$$ where $x \in \mathbb{R}$, $\phi^{\lambda}$ is a $C^{\infty}$ function $C^2$ close to $\phi^{\lambda}_{\tau}$  such that the previous lemma holds for both functions.
\end{definition}

\begin{definition}
Now we define $\alpha: U \to \mathbb{R}$ as $$\alpha(t,x) =  \sum_{k=r+1}^{n-1}x^2_{k} \Phi_k(x),$$ such that for $k=1,\ldots,n-1$, the function $\Phi_k: U \to \mathbb{R}$ is $$\Phi_k(x) = \frac{1}{4} \phi_{k,1}(x_1) \phi_{k,2}(x_2) \phi_{k,3}(x_3) \ldots \phi_{k,2n-1}(x_{2n-1}).$$
\end{definition}

\begin{lemma}\label{estimates}
For $\alpha : U \to \mathbb{R}, (t,x) \to \sum_{k=r+1}^{n-1} x^2_k \Phi_k(x)$, there exists $M_0$ independent of $\epsilon$, the following inequalities are satisfied:
\begin{itemize} 
\item[i.] $|\alpha| \leq M_0 \epsilon^4$, 
\item[ii.] $|\partial_{x_j} \alpha | \leq M_0 \epsilon^2$, 
\item[iii.] $|\partial^2_{x_i x_j} \alpha | \leq M_0 \epsilon$, if $i \neq j$, or if $i \leq r$, or $j \leq r$,
\item[iv.] $| \partial^2_{x_k x_k} \alpha | \leq M_0$,  $k=r+1,\ldots,n-1$.
\end{itemize}
\end{lemma}

\begin{proof} We check each term of the sum $x^2_k \Phi_k$, and observe that $x_k$ is of order $\epsilon^2$, $\Phi_k$ is of order of $1$, $d\Phi_k$ is of order $\epsilon^{-2}$, $d^2\Phi_k$ is of order $\epsilon^{-4}$ and therefore the next inequalities hold: 
\begin{itemize}
\item[i.] $|\alpha| \leq \frac{1}{4} \epsilon^4$.
\item[ii.] $|\partial_{x_j} \alpha | \leq \frac{1}{4} \epsilon^4 2 \epsilon^{-2}$. 
\item[iii.:] $|\partial^2_{x_j x_i} \alpha| \leq \frac{n}{4} \epsilon^4 4 \epsilon^{-2}$ if $j \neq i$. 
\item[iv.:]$| \partial^2_{x_k x_k} \alpha | \leq \frac{1}{4} \epsilon^4 3 \epsilon^{-4} \leq 1$.
\end{itemize}
\end{proof}

By definition $\phi_{k,1}$, $\ldots \phi_{k,n-1}$, except $\phi_{k,k}$, have support on $[-\epsilon,\epsilon]$, and $\phi_{k,i}(0) = 1$, $\phi_{k,i}(\pm \epsilon) = 0$, with $\epsilon < \epsilon_0$, and $\phi_{k,k}$ have support on $[-\epsilon^2,\epsilon^2]$, $\phi_{k,k}(0) = -1$ and $\phi_{k,k}(\pm \epsilon^2) = 0$. This ensures that the only second order partial derivative of $\alpha$ that does not go to $0$ as $\epsilon \to 0$ is $\partial^2_{k,k} \alpha$. Moreover, $\alpha$ is $C^1$-close to the constant zero function.

\begin{remark}\label{o. curvature tensor} The coordinates of the curvature tensor in the tubular neighborhood of $\gamma$ are:
\begin{equation}\label{eq1}
R_{ijkl} = - \frac{1}{2} (\partial^2_{ik} g_{jl} + \partial^2_{jl} g_{ik} - \partial^2_{il} g_{jk} - \partial^2_{jk} g_{il} ) - \Gamma^T_{ik} g^{-1} \Gamma_{jl} + \Gamma^T_{il} g^{-1} \Gamma_{jk},
\end{equation}
\noindent where $\Gamma_{ik} := [\Gamma_{j,ik}]_{j}$ and $\Gamma_{j,ik} := \frac{1}{2} (\partial_i g_{jk} + \partial_k g_{ij} - \partial_j g_{ik}).$

Using that  $\alpha$ is $C^1$-close to the constant zero function, we can get the following estimates for  the curvature tensor of $g^*$:
\begin{eqnarray*}
R^*_{ijkl} (t,x) \approx R_{ijkl}(t,x) & - & \frac{1}{2} (\delta _{j+l,0} \partial^2_{ik} \alpha(t,x) + \delta _{i+k,0} \partial^2_{jl} \alpha(t,x) \\ & - & \delta _{j+k,x} \partial^2_{il} \alpha(t,x) - \delta _{i+l,x} \partial^2_{jk} \alpha(t,x) ),
\end{eqnarray*}
\noindent and so
$$R^*_{0j0l}(t,x) \approx R_{0j0l}(t,x) - \frac{1}{2} (\partial^2_{jl} \alpha(t,x)).$$
\noindent Then:
$$R^*_{0i0j}(t,x) \approx R_{0i0j}(t,x), i \neq j,i,j = 2,\ldots,n-1,$$ 
\begin{eqnarray*} R^*_{0k0k}(t,x) & \approx & R_{0k0k}(t,x) - \frac{1}{2} (\partial^2_{kk} \alpha(t,x)) \\ & \approx & R_{0k0k}(t,x) - \frac{1}{4}(x_k^2 \phi_{k,k}''(x_k) + 4x_k \phi_{k,k}'(x_k) + 2 \phi_{k,k}(x_k)) . \end{eqnarray*}

\end{remark}

Previous remark shows that the curvature is only deformed in the direction of the subspace generated by $\frac{\partial}{\partial x_k},k=r+1,\ldots,n-1$ along geodesics close to $\gamma$. To accomplish this we have constructed a bump function such that, as $\epsilon \to 0$, only the term $\partial^2_{x_k x_k} \alpha,k=r+1,\ldots,n-1$ perturbs the curvature. In particular,  if the curvature is changed by $\frac{1}{4}$ along the closed geodesic $\gamma$, then the curvature is deformed by $\pm \frac{1}{2}$ in the weak directions of the splitting of the geodesic flow, so the curvature for the strong directions is still greater than in the other directions. This explains in a rough way why the geodesic flow still preserves the strong directions.

\subsection{Partial hyperbolicity of the geodesic flow of $g^*$}
\label{s.geode}

To prove that the geodesic flow of the new metric $g^*$ is partially hyperbolic we are going to define the strong stable and strong unstable cones of the geodesic flow of $g^*$.

\begin{definition}\label{def.cones ph}
The strong unstable and strong stable cone fields for $g^*$ are:
\begin{eqnarray*} C^u(v,c) := \left\{ (\xi,\eta) \in S(x,v) : \frac{g^*(\xi_A+\eta_A,\xi_A+\eta_A)}{g^*(\xi,\xi)+g^*(\eta,\eta)} \geq c \right\}, \end{eqnarray*}
\begin{eqnarray*} C^s(v,c) := \left\{ (\xi,\eta) \in S(x,v) : \frac{g^*(\xi_A-\eta_A,\xi_A-\eta_A)}{g^*(\xi,\xi)+g^*(\eta,\eta)} \geq c \right\}. \end{eqnarray*}
\noindent for a real number $c \in (1,2)$, and $v \in T_xM$, $g^*_x(v,v) = 1$.
\end{definition}

\begin{remark}
Notice that the cone field defined above coincides with the cone field associated with $g$ outside the region of the deformation of the metric $g$.
\end{remark}

\begin{remark}
Remember that $$\xi_A := Pr_A \xi, \xi_B := Pr_B \xi,$$ $$\xi'_A = Pr_A \frac{D}{dt} \xi, \xi'_B = Pr_B \frac{D}{dt} \xi.$$ We also define $$\xi_{A'} = \left( \frac{D}{dt} Pr_A \right) \xi, \xi_{B'} = \left( \frac{D}{dt} Pr_B \right) \xi,$$ because for $g^*$ the subspaces $A$ and $B$ are not parallel. 
\end{remark}

\begin{proposition}
\label{p.phcone}
The geodesic flow of $g^*$ preserves the strong unstable cone field $C^u(v,c)$ and the strong stable cone field $C^s(v,c)$ provided by definition \ref{def.cones ph}, for some $c \in (1,2)$ and some $\epsilon$ small enough.
\end{proposition}

We only prove proper invariance of the strong unstable cone (see remark \ref{r.reverse}). We divide the proof in several steps, but first, in the next subsection, we prove that along the geodesic $\gamma$, the geodesic flow of $g^*$ is partially hyperbolic but not hyperbolic.

\subsubsection{Along $\gamma$ the geodesic flow of $g^*$ is not hyperbolic}
\label{s.geod nh}

From a corollary in Eberlein's article \cite{E1} follows:

\begin{coron}[3.4 \cite{E1}] 
If the geodesic flow is Anosov, then the following holds: Let any $\gamma$ be a unit speed geodesic, and $E(t)$ any non-zero perpendicular parallel vector field along $\gamma$, then the sectional curvature $K(\gamma',E)(t) < 0$ for some real number $t$.
\end{coron}

For the geodesic flow of the new metric $g^*$, if   we can find 
$E(t)$  a non-zero perpendicular parallel vector field along $\gamma$, and $K(\gamma',E)(t) = 0$, then the geodesic flow of the metric $g^*$ is not Anosov.

\begin{lemma}
\label{l.eber}
If $\Phi_k(t,0) = -\frac{1}{4}$ then, following Eberlein's criterion, the geodesic flow of $g^*$ is not Anosov.
\end{lemma}

\begin{proof}

Recalling remark \ref{o. curvature tensor}, follows that  the curvature tensor at $\gamma$ is:

\begin{eqnarray*}
R^*_{ijkl} (t,0) = R_{ijkl}(t,0) & - & \frac{1}{2} (\delta _{j+l,0} \partial^2_{ik} \alpha(t,0) + \delta _{i+k,0} \partial^2_{jl} \alpha(t,0) \\ & - & \delta _{j+k,0} \partial^2_{il} \alpha(t,0) - \delta _{i+l,0} \partial^2_{jk} \alpha(t,0) ),
\end{eqnarray*}
\noindent and
$$R^*_{0j0l}(t,0) = R_{0j0l}(t,0) - \frac{1}{2} (\partial^2_{jl} \alpha(t,0)).$$
\noindent Then, along $\gamma$:
$$R^*_{0i0j}(t,0) = R_{0i0j}(t,0), i \neq j,i,j = 2,\ldots,n-1,$$ 
\begin{eqnarray*} R^*_{0k0k}(t,0) & = & R_{0k0k}(t,0) - \frac{1}{2} (\partial^2_{kk} \alpha(t,0)) \\ & = & R_{0k0k}(t,0) - \Phi_k(t,0). \end{eqnarray*}
\noindent For the initial metric and $k = r+1,\ldots,2n-1$: 
$$R_{0k0k}(t,0) = g_{00}(t,0)g_{kk}(t,0)K(\gamma'(t),e_k(t)) = - \frac{1}{4}.$$
\noindent So, if $\Phi_k(t,0) = - \frac{1}{4}$, then $R^*_{0k0k}(t,0) = 0$. Then, Eberlein's corollary applies, and the geodesic flow of $g^*$ is not Anosov. 
\end{proof}

\subsubsection{Along $\gamma$ the geodesic flow of $g^*$ is partially hyperbolic}
\label{s.phgamma}

We are going to show that the strong unstable cone field of the geodesic flow of section \ref{s.orig} still works for the geodesic flow of the new metric $g^*$ along $\gamma$.

\begin{lemma}
\label{l.geod}
For the new metric $g^*$ and along the geodesic $\gamma$ there is an invariant splitting $S(t) = E^{ss} \oplus E^c \oplus E^{uu}$, such that $E^{ss} = E_g^{ss}$, $E^c = E_g^s \oplus E_g^u$, $E^{uu} = E_g^{uu}$, where $E^{\sigma}_g$ are the subbundles of the hyperbolic invariant splitting of the geodesic flow of the original metric $g$, $\sigma = uu, u, s, ss$, and $S(t)$ is the contact structure of $U^*M$ along $(\gamma(t),\gamma'(t))$.   
\end{lemma}

\begin{proof}
The normal coordinates that were defined for $g$ along the closed geodesic $\gamma$ are still normal
coordinates for $g^*$, and observe that it has the same Christoffel symbols along $\gamma$. This implies that $g^*$ has the same parallel transport as $g$ along $\gamma$. 

Taking $\{E_0(t) = \gamma'(t),E_1(t),\ldots,E_r(t),E_{r+1}(t),\ldots,E_{n-1}(t) \}$,  an orthonormal basis of parallel vector fields in $T_{\gamma(t)}M$, then  $\zeta(t) = \sum_{i=0}^{2n-1} f_i(t) E_i(t)$ are Jacobi fields along $\gamma$ if they are the solutions of the following equation:
\begin{eqnarray*} 0 & = & \zeta''(t) + R^*(\gamma'(t),\zeta(t))\gamma'(t) \\ & = & \sum_{i,j=0}^{2n-1} (f''_i(t) + R^*(E_0,E_j,E_0,E_i)(t)f_i(t)) E_i(t), \end{eqnarray*} which implies that
$$ 0 = f''_i(t) + \sum_{j=1}^{2n-1} R^*(E_0,E_j,E_0,E_i)(t)f_i(t), i=0,\ldots,2n-1,$$ which is equivalent to
$$ \begin{bmatrix} f(t) \\ f'(t) \end{bmatrix}' = \begin{bmatrix} 0 & I \\ - K^*(t) & 0 \end{bmatrix} \begin{bmatrix} f(t) \\ f'(t) \end{bmatrix},$$
$$K^*_{ij}(t) := R^*(E_0,E_j,E_0,E_i)(t).$$
\noindent Along $\gamma$ we have:
$$K^*(t) = \begin{bmatrix} - Id_{r} & 0 \\ 0 & 0 \end{bmatrix}.$$
\noindent The hyperbolic subbundles are $E^{uu}$, spanned by $(e^t e_i(t),e^t e_i(t))$, $i = 1,\ldots,r$ and $E^{ss}$, spanned by $(e^{-t} e_i(t),- e^{-t} e_i(t))$, $i = 1,\ldots,r$ and $E^{ss}$, the same as for the metric $g$. And there is a central direction spanned by the Jacobi fields related to the curvature $K(\gamma'(t),E_k(t))$, $E_k(t)$ and $t E_k(t)$, for $k= r+1,\ldots,2n-1$. This implies we have a central bundle $E^c$ along the geodesic $\gamma$. Notice that $\{e_k(t)\}_{k=r+1}^{2n-1}$ and $\{E_k(t)\}_{k=r+1}^{2n-1}$ generate the same subspace of $T_{\gamma(t)}M$, invariant by parallel transport because it is orthogonal to $\gamma'(t)$ and $A(\gamma(t),\gamma'(t))$. Then $E^c = E^s_g \oplus E^u_g$.
\end{proof}

\subsubsection{Preservation of the cone field for parallel vectors}
\label{s.para1}

Now we adapt to the geodesic flow of the new metric $g^*$,  the same type of calculations done in  subsection \ref{s. cone var for anosov}. To prove the partial hyperbolicity of this new flow, we divided the set of vectors whose geodesics cross the neighborhood where we change the original metric. First we verify the proper invariance of the cone field for parallel vectors (see definition in the beginning of section \ref{s.geode}).

By the formula of the bump function $\Phi_k$ we have that, as $\epsilon$ goes to zero, the partial derivatives of second order of $\alpha$ which do not involve the direction of $\frac{\partial}{\partial_{x_k}}$ go to zero. The only one that does not go to zero as $\epsilon \to 0$ is $\partial^2_{k,k} \Phi_k$. So, the following holds: 
$$R^*_{010k} \approx R_{010k}, k = 2,\ldots,n-1,$$ 
 $$R^*_{0k0k} \approx R_{0k0k} - \frac{1}{2} \partial^2_{k,k} \alpha $$
\noindent If $v = (v_0,0,\ldots,0)$ then: 
$$R^*(v,\xi,v,\eta) \approx R(v,\xi,v,\eta) - \frac{1}{2} \partial^2_{\xi \eta} \alpha v^2_0$$ 
$$\approx R(v,\xi,v,\eta) - \frac{1}{2} \sum_{k=r+1}^{n-1}\partial^2_{kk} \alpha v^2_0 \xi_k \eta_k.$$
\noindent When we use the symbol $\approx$ we mean that the difference between the left side and the right side is of order $\epsilon$. It depends on the size of $|\alpha|,|\partial \alpha|,|\partial^2_{ij}\alpha|, i \neq j$, and the size of $supp(\Phi_i), i=r+1,\ldots,n-1$ (see lemma \ref{estimates}).

\begin{remark}
Remember that for the geodesic flow of section \ref{s. cone var for anosov} $\xi_{A'}$ and $\eta_{A'}$ do not appear on the calculations due to the fact that $A$ is a parallel subspace along geodesics (see equation \ref{eq:jacobi}). For the new metric $g^*$ this is not the case but, as $A'$ depends on the first derivative of $\alpha$ they appear as a small term of  perturbation.
\end{remark}

\begin{lemma}
\label{l.para1}
For parallel vectors the angle cone variation is positive (therefore, from subsection \ref{s. cone variation def}, the cone is closed by the action of the derivative).
\end{lemma}

\begin{proof}

We begin by approximating the angle cone variation at parallel vectors with respect to the derivative of the geodesic flow by an expression that is better to work with. This expression is equal to the one for the geodesic flow of $g$ (recall subsection \ref{s. cone var for anosov}) except for the term related to the second derivative of $\alpha$ and $\xi_k$, $\eta_k$ related to the central direction along $\gamma$. Also recall  remark \ref{l.para1} about $\xi_{A'}$ and $\eta_{A'}$.

\begin{eqnarray*} & & \frac{d}{dt} \frac{g^*(\xi_A+\eta_A,\xi_A+\eta_A)}{g^*(\xi,\xi) + g^*(\eta,\eta)} = 2 \frac{g^*(\xi_A+\eta_A,\xi_A+\eta_A)}{(g^*(\xi,\xi) + g^*(\eta,\eta))^2} \left(\frac{5}{8}g^*(\xi - \eta,\xi - \eta) \right. \\ & & \left. + \frac{3}{8}g^*(\xi,\xi) - \frac{3}{4}g^*(\xi_A,\eta_A) + \frac{1}{2} \sum_{k=r+1}^{n-1} \partial^2_{kk} \alpha v^2_0 \xi_k \eta_k + \frac{3}{8} g^*(\eta,\eta) \right) \\ & & = 2 \frac{g^*(\xi_A+\eta_A,\xi_A+\eta_A)}{(g^*(\xi,\xi) + g^*(\eta,\eta))^2} \left( \left(\frac{g^*(\xi_A+\eta_A,\xi_{A'}+\eta_{A'})}{g^*(\xi_A+\eta_A,\xi_A+\eta_A)} \right)(g^*(\xi,\xi) \right. \\ & & + g^*(\eta,\eta)) - \left(\frac{g^*(\xi_A+\eta_A,\xi_A+R^*(v,\xi)v)}{g^*(\xi_A+\eta_A,\xi_A+\eta_A)} \right) (g^*(\xi,\xi) + g^*(\eta,\eta)) + \\ & & \left. \frac{1}{4} g^*(\xi,\eta) + \frac{3}{4}g^*(\xi_A,\eta_A) - \frac{1}{2} \sum_{k=r+1}^{n-1} \partial^2_{kk} \alpha v^2_0 \xi_k \eta_k + R^*(v,\xi,v,\eta) \right). \end{eqnarray*}
\noindent We define as $\xi_{A'}$ the covariant derivative of the projection to $A$ applied to $\xi$: $(\nabla^* Pr_A) \xi$. If $c$ is the opening of the cone (see subsection \ref{s.ph and cones}) and $g^*(\xi,\xi) + g^*(\eta,\eta) = 1$, because the derivative does not depend on the norm of the $(\xi,\eta)$, the equation above is:
$$= 2c (c^{-1}(g^*(\xi_A+\eta_A,\xi_{A'}+\eta_{A'})) - c^{-1}(g^*(\xi_A+\eta_A,\xi_A+R^*(v,\xi)v)) $$
$$+ \frac{1}{4} g^*(\xi,\eta) + \frac{3}{4}g^*(\xi_A,\eta_A) - \frac{1}{2} \sum_{k=r+1}^{n-1} \partial^2_{kk} \alpha v^2_0 \xi_k \eta_k + R^*(v,\xi,v,\eta)).$$
\noindent Then:
\begin{eqnarray*} \left| g^*(\xi_A+\eta_A,\xi_A + R^*(v,\xi)v) \right|  & \leq & \left|g^*(\xi_A+\eta_A,R^*(v,\xi)v - R(v,\xi)v) \right| \\ & + & \left| g^*(\xi_A+\eta_A,\xi_A + R(v,\xi)v) \right| . \end{eqnarray*}
\noindent Since $\left| g^*(\xi_A+\eta_A,\xi_A + R(v,\xi)v) \right|$ depends on $|\alpha|$, and $\left|g^*(\xi_A+\eta_A,R^*(v,\xi)v - R(v,\xi)v) \right| $ depends on $|\alpha|$, $|\partial \alpha|$, and $|\partial^2_{j\xi} \alpha|$, $j=1,\ldots,r$ and these terms are limited by $M \epsilon$ (recall lemma \ref{estimates}), we can say that, for some big enough $M_1$ independent of $\epsilon$:
\begin{eqnarray*} & & \left| g^*(\xi_A+\eta_A,\xi_A + R^*(v,\xi)v) \right|  \leq  \left|g^*(\xi_A+\eta_A,R^*(v,\xi)v - R(v,\xi)v) \right| + \\ & &  \left| g^*(\xi_A+\eta_A,\xi_A + R(v,\xi)v) \right|  \leq M_1 \epsilon. \end{eqnarray*}
\noindent For the same reasons:
$$\left| g^*(\xi_A+\eta_A,\xi_{A'}+\eta_{A'}) \right| \leq M_0 \left\| g^* - g \right\|_{C^1} (\left| \xi \right|^* + \left| \eta \right|^*) \leq M_1 \epsilon.$$
\begin{eqnarray*} \left| \frac{1}{4} g^*(\xi,\eta) + \frac{3}{4}g^*(\xi_A,\eta_A) - \frac{1}{2} \sum_{k=r+1}^{n-1} \partial^2_{kk} \alpha v^2_0 \xi_k \eta_k + R^*(v,\xi,v,\eta) \right| \leq M_1 \epsilon. \end{eqnarray*}
\noindent Suppose $M_1$ sufficiently big to be the same in the three inequalities above. So we have:
\begin{eqnarray*} & & \left| \frac{d}{dt} \frac{g^*(\xi_A+\eta_A,\xi_A+\eta_A)}{g^*(\xi,\xi) + g^*(\eta,\eta)} - 2\frac{g^*(\xi_A+\eta_A,\xi_A+\eta_A)}{(g^*(\xi,\xi) + g^*(\eta,\eta))^2} \left(\frac{5}{8}g^*(\xi - \eta,\xi - \eta) + \right. \right. \\ & & \left. \left. \frac{3}{8}g^*(\xi,\xi) - \frac{3}{4}g^*(\xi_A,\eta_A) + \frac{1}{2} \sum_{k=r+1}^{n-1} \partial^2_{kk} \alpha v^2_0 \xi_k \eta_k + \frac{3}{8} g^*(\eta,\eta)\right) \right| \\ & & \leq 2c(3M_1)\epsilon = M_2 \epsilon. \end{eqnarray*}
\noindent Let us analyses the following expression over the initial closed geodesic: 
$$\left(\frac{3}{8}g^*(\xi,\xi) - \frac{3}{4}g^*(\xi_A,\eta_A) + \frac{1}{2} \sum_{k=r+1}^{n-1} \partial^2_{kk} \alpha v^2_0 \xi_k \eta_k + \frac{3}{8} g^*(\eta,\eta)\right) = $$
$$\frac{3}{8}(\xi_1^2 + \xi_2^2 + \ldots + \xi_{n-1}^2 + \eta_1^2 + \eta_2^2 + \ldots + \eta_{n-1}^2 - 2 \sum_{k=1}^r \xi_k \eta_k + \frac{4}{3} \sum_{k=r+1}^{n-1} \partial^2_{kk} \alpha v^2_0 \xi_k \eta_k).$$
\noindent The expression $\xi_1^2 + \eta_1^2 + \xi_2^2 + \eta_2^2 + \ldots + \xi_{n-1}^2 + \eta_{n-1}^2 - 2 \xi_1 \eta_1 - \ldots - 2 \xi_r \eta_r + \frac{4}{3} \sum_{k=r+1}^{n-1} \partial^2_{kk} \alpha v^2_0 \xi_k \eta_k$ is equal to $\sum_{k=1}^r(\xi_k - \eta_k)^2 + \sum_{k=r+1}^{n-1} \xi_k^2 - \frac{2}{3} \xi_k \eta_k + \eta_k^2 = \sum_{i=k}^r (\xi_k - \eta_k)^2 + \sum_{k=r+1}^{n-1} (\xi_k - \frac{1}{3} \eta_k)^2 + \frac{8}{9} \eta^2_k$ which is positive in the boundary of the cone with opening $c$. This implies that along the closed geodesic $\gamma$ the cone is preserved, but that we already knew. We need to prove the positivity of the derivative along the other geodesics of the flow. So, we need the following:
$$\inf_{a \in [-1 - \frac{\delta}{2},1 + \frac{\delta}{2}]} \inf \{ \sum_{k=r+1}^{n-1} \xi_k^2 - \frac{4a}{3} \xi_k \eta_k + \eta_k^2 \} \geq L(a,b) > 0,$$ 
\noindent for any $(\xi,\eta)$ in the boundary of the cone with opening $c \in [a,b] \subset (1,2)$.
Because $g^*$ is a $C^{\infty}$ metric, and its coordinates along $\gamma$ are $\delta_{ij}$, if the neighborhood of $\gamma$ is sufficiently small, if $\epsilon$ is small enough, we can conclude:
\begin{eqnarray*} \inf_{x \in supp(\alpha)} \inf \{ (g^*(\xi,\xi) - 2g^*(\xi_A,\eta_A) + \frac{4}{3} \sum_{k=r+1}^{n-1} \partial^2_{kk} \alpha v^2_0 \xi_k \eta_k + g^*(\eta,\eta)) \} \geq \frac{1}{2} L(a,b) > 0. \end{eqnarray*}
\noindent So:
\begin{eqnarray*} & & \inf_{x \in supp(\alpha)} \inf \{ \frac{3}{8}g^*(\xi,\xi) - \frac{3}{4}g^*(\xi_A,\eta_A) + \frac{1}{2} \sum_{k=r+1}^{n-1} \partial^2_{kk} \alpha v^2_0 \xi_k \eta_k + \frac{3}{8} g^*(\eta,\eta) \} \geq \\ & & = \frac{3}{16} L(a,b) > 0. \end{eqnarray*} 
\noindent This implies that, if $\epsilon < \frac{3}{32 M_2} L(a,b)$, for $(\xi,\eta)$ in the boundary of the cone with opening $c \in [a,b] \subset (1,2)$, and for $v = (v_0,0,\ldots,0)$, then the derivative of equation (\ref{eq.cone}) is positive.

\end{proof}

\subsubsection{Extension of the cone property to $\theta$-parallel vectors}
\label{s.para2}

Now we are going to show that this derivative is positive not only for parallel vectors ($v = (v_0,0,\ldots,0)$), but for $\theta$-parallel vectors. 

\begin{lemma}
\label{l.para2}
For $\theta$-parallel vectors the angle cone variation is positive (therefore, from subsection \ref{s. cone variation def}, the cone is closed by the action of the derivative).
\end{lemma}

\begin{proof}

\begin{eqnarray*} R^*(v,\xi,v,\eta) - R(v,\xi,v,\eta) \approx - \frac{1}{2} \sum_{k=r+1}^{n-1} \partial^2_{kk} \alpha (v_k^2 \xi_0 \eta_0 + v_0^2 \xi_k \eta_k - v_0 v_k (\xi_0 \eta_k + \xi_k \eta_0)). \end{eqnarray*}
\noindent This is so because (\ref{eq1}) implies the following relation:
\begin{equation}\label{eq2}
R^*_{ijkl} - R_{ijkl} \approx - \frac{1}{2} (\partial^2_{ik} \Delta g_{jl} + \partial^2_{jl} \Delta g_{ik} - \partial^2_{il} \Delta g_{jk} - \partial^2_{jk} \Delta g_{il}),
\end{equation}
\noindent where $\approx$ means that the equation depends on $\alpha$ and $\partial \alpha$, and $\Delta g_{ij} := g^*_{ij} - g_{ij}$. So we can say that:

\begin{eqnarray*} \left| R^*(v,\xi,v,\eta) - R(v,\xi,v,\eta) + \frac{1}{2} \sum_{k=r+1}^{n-1} \partial^2_{kk} \alpha v^2_0 \xi_k \eta_k \right| \leq M_1 \epsilon + M_0 |\theta| (\|\xi\|^* \|\eta\|^*). \end{eqnarray*}
\noindent So, for the derivative we have:
\begin{eqnarray*} & & \left| \frac{d}{dt} \frac{g^*(\xi_A+\eta_A,\xi_A+\eta_A)}{g^*(\xi,\xi) + g^*(\eta,\eta)} - 2\frac{g^*(\xi_A+\eta_A,\xi_A+\eta_A)}{(g^*(\xi,\xi) + g^*(\eta,\eta))^2} \left(\frac{5}{8}g^*(\xi - \eta,\xi - \eta) + \right. \right. \\ & & \left. \left. \frac{3}{8}g^*(\xi,\xi) - \frac{3}{4}g^*(\xi_A,\eta_A) + \frac{1}{2} \sum_{k=r+1}^{n-1} \partial^2_{kk} \alpha v^2_0 \xi_k \eta_k + \frac{3}{8} g^*(\eta,\eta)\right) \right| \\ & & 
\leq M_2 \epsilon + M_0|\theta| (\|\xi\|^* \|\eta\|^*). \end{eqnarray*}
\noindent So, if we calculate for $(\xi,\eta)$ in $g^*(\xi,\xi) + g^*(\eta,\eta) = 1$, we have that if $|\theta| < \frac{3}{64 M_0} L'(a,b)$ and $\epsilon < \frac{3}{32 M_2} L(a,b)$, then:
$$ \frac{d}{dt} \frac{g^*(\xi_A+\eta_A,\xi_A+\eta_A)}{g^*(\xi,\xi) + g^*(\eta,\eta)} \geq \frac{3}{32}L(a,b) > 0$$
\noindent Then we conclude that, in the band $\{ (x,v) \textrm{ is } \theta \textrm{-parallel to } \gamma \}$ the cones are properly invariant for the geodesic flow.

\end{proof}

\subsubsection{The control of the cones for $\theta$-transversal vectors}
\label{s.transv}

For vectors that are not $\theta$-close to $(v_0,0,\ldots,0)$ i.e. that are $\theta$-transversal to $\gamma$, we do not have preservation of the cones. To overcome this difficulty we choose an $\epsilon$ small enough such that the cone with opening $b$ stays inside the cone with opening $a$. This can be done since $\alpha$ is $C^1$ close to zero, the second derivative of $\alpha$ is uniformly bounded independently of   $\epsilon$. So:
$$\frac{d}{dt} \frac{g^*(\xi_A+\eta_A,\xi_A+\eta_A)}{g^*(\xi,\xi) + g^*(\eta,\eta)} \geq M.$$
\noindent Observe that as $\epsilon$ goes to $0$, the support of the deformation of the metric shrinks. As it shrinks, the time that the geodesics take to cross this neighborhood of the geodesic $\gamma$ goes to zero. So, as we can control the time which these geodesics spend inside the tubular neighborhood $U(\epsilon)$ of the geodesic $\gamma$, we choose an $\epsilon$ such that the cone with opening $b$ stays inside the cone of opening $a$.

Let us be more precise:

\begin{lemma}
\label{l.transv}
The time which a $\theta$-transversal geodesics cross the tubular  neighborhood  $U(\epsilon)$ of the deformation of the metric $g$ is comparable to $\epsilon$.
\end{lemma}

\begin{proof}
To see that the time spent inside $U(\epsilon)$  is comparable to $\epsilon$ we need to express the geodesic vector field in Fermi coordinates of the neighborhood. In fact, we can use that coordinates since  we don't need the coordinates in the whole neighborhood of the closed geodesic $\gamma$. The maps $d\pi$ and $K$ in the Fermi's coordinates are given by:
$$d\pi \xi = (\xi_0,\xi_1,\ldots,\xi_{2n-1}),$$
$$K \xi = \left(\xi_{2n+k} + \sum_{i,j=0}^{2n-1} \Gamma^{*k}_{ij} v_i \xi_j \right)_{k=0}^{2n-1}.$$
\noindent So, the pre-image of $(v,0)$ by the map $(d\pi,K)$ is:
$$\left(v_0,v_1,\ldots,v_{2n-1},- \sum_{i,j=0}^{2n-1} \Gamma^{*0}_{ij} v_i v_j,- \sum_{i,j=0}^{2n-1} \Gamma^{*1}_{ij} v_i v_j,\ldots,- \sum_{i,j=0}^{2n-1} \Gamma^{*2n-1}_{ij} v_i v_j \right).$$
\noindent Since $g^*$ is $C^{\infty}$ and along the geodesic $\gamma$, $\Gamma^{*k}_{ij} = 0$, then, if $\epsilon$ is sufficiently small, the geodesic vector field is approximately $(v_0,v_1,\ldots,v_{2n-1},0,0,\ldots,0)$.

Since the second part of the geodesic vector field is small as $\epsilon$ is small, we can say that geodesics such that $|v_i| \geq \theta$ for some $i = 1,\ldots,2n-1$ cross the neighborhood in at most time $\frac{\epsilon}{\theta}$, and they arrive to the complement of  $\{ v \in U^*M : |v_i| < \frac{\theta}{2}, i=1,2,\ldots,2n-1 \}$.
\end{proof}

\subsubsection{Proof of proposition \ref{p.phcone}}
\label{s.phcone}

Based on previous  lemmas about paralell and transversal geodesic and lemma \ref{l.transv} we can conclude the proof of  proposition \ref{p.phcone}:

\noindent{\em Proof of  proposition \ref{p.phcone}.} First, take an orbit of the geodesic flow of $g^*$. If it never crosses the region of the deformation, where $g^*$ equals the original metric $g$, then the cone field is preserved. If it crosses the region of deformation, then it takes some time $T'$ inside this region. If it is $\theta$-parallel to the geodesic $\gamma$, then it preserves the cone field (lemma \ref{l.para2}). If it turns, after this time $T'$, into a $\theta$-transversal geodesic, then it spends $T' + k \epsilon$ time inside this region (lemma \ref{l.transv}), and then it leaves it and spend some time outside it. As the set of the orbits which leave this region is a compact set, the infimum is positive. Let us say they spend at least $T_{\epsilon}$ outside the neighborhood. As $\epsilon$ goes to zero, $T_{\epsilon}$ does not goes to zero. If it did, we could get a sequence of geodesics outside $\{ v \in U^*M : |v_i| < \frac{\theta}{2}, i=1,2,\ldots,2n-1 \}$ which would spend a small   time outside the neighborhood $U(\epsilon)$ of $\gamma$ before enter it again. So, in the limit, there would be a contradiction with the uniqueness of the solutions of the ordinary differential equations of the geodesic flow. So the time spent outside the neighborhood of $\gamma$ is bounded from below - let us say it is bounded from below by $T$. This means that we can choose $\epsilon$ so that the quotient between the time spent inside and the time spent outside of the neighborhood of $\gamma$ is as small as we want. As small as it is necessary for the preservation of the strong unstable and strong stable cones. So, the orbit spends some time $k \epsilon$ where there is a little expansion of the angle of the cone field, then spends time at least $T$ in the region where there's contraction of the angle of the cone field.

Outside the neighborhood of the deformation the following holds:

\begin{eqnarray*} \frac{d}{dt} \frac{(g^*(\xi_A+\eta_A,\xi_A+\eta_A)}{g^*(\xi,\xi)+g^*(\eta,\eta)} = \frac{d}{dt} \frac{g(\xi_A+\eta_A,\xi_A+\eta_A)}{g(\xi,\xi)+g(\eta,\eta)} \geq \frac{3}{8}c(2-c), \end{eqnarray*}
\noindent for $(\xi,\eta)$ in the boundary of the cone of opening $c$. So, for cones with boundary in $[a,b]$, we have:
$$\frac{d}{dt} \frac{g^*(\xi_A+\eta_A,\xi_A+\eta_A)}{g^*(\xi,\xi)+g^*(\eta,\eta)^*} \geq \frac{3}{8}b(2-b).$$

So we choose $a'$ such that $|a'-b| < \frac{3}{16}b(2-b)T$. This ensures that outside the neighborhood the geodesic flow sends the cone with opening $a'$ inside the cone with opening $B$ in time $\frac{T}{2}$. For $\epsilon$ sufficiently small, with the inferior limit of the derivative not depending on $\epsilon$, the cone with opening $b$ is not sent outside the cone with opening $a'$. 

So, we have preservation of the cone field, although there is a region where the cone field is not properly invariant, because the orbits of length $T$ of the geodesic flow cross this region in an interval of time as small as we want. So the preservation of the cone field holds because after that it takes an interval of length $\frac{T}{2}$ for the cones to be properly contained. 
\qed

\subsubsection{Exponential growth of the Jacobi fields}
\label{s.expo}

So, the strong unstable cone is preserved by the new geodesic flow. By reversibility of geodesic flows, the strong stable cone is preserved too. But preservation of these cones only proves that there are invariant subbundles with domination. We have to show that there is exponential growth along these strong directions. 

\begin{proposition}
\label{p.expo}
For the geodesic flow of $g^*$ there is exponential expansion of vectors in $C^u(v,c)$. 
\end{proposition}

\begin{proof}
The geodesic flow has an invariant splitting of the following kind: $S(UM) = E^{cs} \oplus E^c \oplus E^{cu}$. Because it is symplectic, and $dim(E^{cs}) = dim(E^{cu})$, we apply the lemma \ref{l.contreras}. Then the invariant subbundles $E^{cs}$ and $E^{cu}$ are hyperbolic, or there is exponential contraction of vectors in the former subbundle, and exponential expansion of vectors in the later subbundle. 
\end{proof}

\subsection{Finishing the  proof of theorem A, B, and corollaries C.1 and C.2}
\label{s.conc}

Summarizing, in proposition \ref{p.phcone} we proved the proper invariance of the unstable and stable cones and in proposition \ref{p.expo} we proved the exponential expansion or contraction respectively. Therefore we conclude:

\begin{theorem}
\label{t.phgeodesic}
Let $(M,g)$ be a K\"ahler manifold of negative holomorphic curvature $-1$ or a quaternion K\"ahler locally symmetric space of negative curvature. Then there is a $C^\infty$ metric $g^*$ on $M$ such that its geodesic flow is partially hyperbolic but not Anosov. Also, $g^*$ is $C^2$-far from the open set of metrics on $M$ which have Anosov geodesic flows.
\end{theorem}
To finish the proof of theorems A and B, we have to show that  some of the deformed metric   are transitive; this is concluded in the next corollary.  
\begin{corollary}
\label{c.cor3}
There is a Riemannian manifold $(M,\widetilde{g})$ such that its geodesic flow is partially hyperbolic, non-Anosov, transitive. Moreover, $(M,\widetilde{g})$ has no conjugate points.
\end{corollary}

\begin{proof}
By a theorem of Eberlein \cite{E4}, if $(M,g)$ is a Riemannian manifold with an Anosov geodesic flow and $(M,\widetilde{g})$ is another one without conjugate points, then the geodesic flow of $\widetilde{g}$ is transitive. The set of metrics of $M$ without conjugate points is closed. So, if we consider the one parameter family of metrics $g_s = g + s(g^* - g)$, with $s \in [0,1]$, there is an $s_0$ such that the geodesic flow of $g_{s_0}$ has no conjugate points, is partially hyperbolic, and it is not Anosov, which implies immediately that it is transitive.
\end{proof}

Now, we prove corollary C.1. that states that there is an open set of metrics whose geodesic flows are partially hyperbolic and have conjugate points.

\noindent{\em Proof of corollary C.1.}
Ruggiero in \cite {R2} proved that the $C^2$-interior of the set of metrics with no conjugate points is the set of metrics whose geodesic flow is Anosov. So, since the example has a partially hyperbolic geodesic flow which is non-Anosov, there is a metric $C^2$-close to it that has conjugate points. Moreover, it is a corollary of Ruggiero's theorem that the set of metrics with conjugate points is open, so there is a $C^2$-open set of metrics with conjugate points and a partially hyperbolic geodesic flow. 
\qed

Now we provide the proof of corollary C.2 about Hamiltonian flows.

\noindent {\em Proof of corollary C.2.}
For the same reasons of the previous corollary there is an open set of Hamiltonians with the same property, near geodesic Hamiltonians.
\qed

\section{Symmetric spaces of nonpositive curvature}
\label{s.symm}

In this section, first we give a brief introduction to the subject of symmetric and locally symmetric space \cite{E2},\cite{E3},\cite{J}, and later we prove that the geodesic flow of a compact locally symmetric spaces of nonpositive curvature is partially hyperbolic only if it has nonconstant negative curvature.

\begin{definition}
A simply connected Riemannian manifold is called symmetric if for every $x \in M$ there is an isometry $\sigma_x: M \to M$ such that
$$\sigma_x(x) = x, d\sigma_x(x) = - id_{T_xM}.$$

\noindent The property of being symmetric is equivalent to:
\begin{itemize}
\item $\nabla R \equiv 0$,
\item if $X(t)$, $Y(t)$ and $Z(t)$ are parallel vector fields along $\gamma(t)$, then $R(X(t),Y(t))Z(t)$ is also a parallel vector field along $\gamma(t)$.
\end{itemize}
\end{definition}

\begin{remark}
A symmetric Riemannian manifold is geodesically complete and every two points can be connected by a geodesic.
\end{remark}

\begin{definition}
A complete Riemannian manifold with $\nabla R \equiv 0$ is called locally symmetric.
\end{definition}

Each simply connected symmetric space $M$ is the quotient of the Lie group $G$ of isometries of $M$ with a left invariant metric by its maximal compact subgroup $K$: $M = G / K$. Each compact locally symmetric space $N$ is the quotient of a simply connected symmetric space $M$ by a cocompact lattice $\Gamma$ of $G$ acting on $M$ discretely, without fixed points and isometrically, such that $N = M / \Gamma$ \cite{E2},\cite{E3},\cite{J}.

\begin{proposition}
\label{p.symbund}
Let $N$ be a locally symmetric space, $p \in N$, $v \in T_pN$, $c$ a geodesic such that $c(0) = p$, $c'(0) = v$, there are $v_1, \ldots, v_{n-1}$ an orthogonal basis of eigenvectors of $R_{c'(0)}$ orthogonal to $v$ with eigenvalues $\rho_1, \ldots, \rho_{n-1}$, and parallel vector fields $v_1(t), \ldots, v_{n-1}(t)$ along $c$ such that $v_i(0) = v_i$. Moreover, the Jacobi fields along $c$ are linear combinations of the following Jacobi fields 
$$c_{\rho_j}(t)v_j(t) \textrm{ and } s_{\rho_j}(t)v_j(t),$$
\noindent where $$c_{\rho} (t) := \left\{ \begin{aligned} & \cos \sqrt{\rho} t, \rho > 0, \\ & \cosh \sqrt{-\rho} t, \rho < 0, \\ & 1, \rho = 0, \end{aligned} \right. s_{\rho} (t) := \left\{ \begin{aligned} & \frac{1}{\sqrt{\rho}} \sin \sqrt{\rho} t, \rho > 0, \\ & \frac{1}{\sqrt{-\rho}} \sinh \sqrt{-\rho} t, \rho < 0, \\ & t, \rho = 0. \end{aligned} \right. $$ 
\end{proposition}

The proof of the proposition is standard and it relies on the two following facts: $R_{v} : T_pN \to T_pN : w \to R(v,w)v$ is a self-adjoint map and the curvature tensor is parallel \cite{J}.

\begin{definition}
Let $M = G / K$ be a symmetric space, where $G$ is the Lie group of isometries of $M$ and $K$ the maximal compact subgroup of $G$. Let $\mathfrak{g}$ be the algebra of Killing fields on the symmetric space $M$ and $p \in M$. Define
$$\mathfrak{k} := \{ X \in \mathfrak{g} : X(p) = 0 \},$$
$$\mathfrak{p} := \{ X \in \mathfrak{g} : \nabla X (p) = 0 \}.$$
\noindent For these subspaces of $\mathfrak{g}$, $\mathfrak{k} \oplus \mathfrak{p} = \mathfrak{g}$ and $\mathfrak{k} \cap \mathfrak{p} = \{ 0 \}$, and $T_pM$ identifies with $\mathfrak{p}$.
\end{definition}

\begin{remark}
In fact the Lie algebra of $G$ is $\mathfrak{g}$ and the Lie algebra of $K$ is $\mathfrak{k}$.
\end{remark}

\begin{definition}
Given $p \in M$, we define the involution $\phi_p(g): G \to G: g \to \sigma_p \circ g \circ \sigma_p$ where $G$ is a Lie group. Then, we obtain $\theta_p := d \phi_p : \mathfrak{g} \to \mathfrak{g}$. Since $\theta_p^2 = id$ and $\theta_p$ preserves the Lie brackets, the properties of these subspaces of $\mathfrak{g}$ are:

\begin{itemize}
\item[i.] $\theta_{p|\mathfrak{k}} = id$,
\item[ii.] $\theta_{p|\mathfrak{p}} = - id$,
\item[iii.] $[\mathfrak{k},\mathfrak{k}] \subset \mathfrak{k}$, $[\mathfrak{p},\mathfrak{p}] \subset \mathfrak{k}$, $[\mathfrak{k},\mathfrak{p}] \subset \mathfrak{p}$,
\end{itemize}
\end{definition}

\begin{proposition}
With the identification $T_pM \cong \mathfrak{p}$ the curvature tensor of $M$ satisfies
$$R(X,Y)Z(p) = [X,[Y,Z]](p)$$
\noindent for all $X,Y,Z \in \mathfrak{p}$. In particular, $R(X,Y)X(p) = - (ad_X)^2(Y)(p)$. 
\end{proposition}

\begin{remark}
We are going to consider only symmetric spaces with nonpositive sectional curvature.
\end{remark}

Fix a maximal Abelian subspace $\mathfrak{a} \subset \mathfrak{p}$. Let $\Lambda$ denote the set of roots determined by $\mathfrak{a}$, and $$\mathfrak{g} = \mathfrak{g}_0 + \sum_{\alpha \in \Lambda} \mathfrak{g}_{\alpha},$$
\noindent where $\mathfrak{g}_{\alpha} = \{ w \in \mathfrak{g} : (ad X) w = \alpha(X) w, \forall X \in \mathfrak{a} \}$, $\alpha: \mathfrak{a} \to \mathbb{R}$ is a one-form \cite{J}. Notice that the subindexes are the one forms, not their values at each vector in $\mathfrak{\alpha}$.

Define a corresponding decomposition for each $\alpha \in \Lambda$, $\mathfrak{k}_{\alpha} = (id + \theta) \mathfrak{g}_{\alpha}$ and $\mathfrak{p}_{\alpha} = (id - \theta) \mathfrak{g}_{\alpha}$. Then:

\begin{itemize}
\item[i.] $id + \theta: \mathfrak{g}_{\alpha} \to \mathfrak{k}_{\alpha}$ and $id - \theta: \mathfrak{g}_{\alpha} \to \mathfrak{p}_{\alpha}$ are isomorphisms,
\item[ii.] $\mathfrak{p}_{\alpha} = \mathfrak{p}_{- \alpha}$, $\mathfrak{k}_{\alpha} = \mathfrak{k}_{- \alpha}$, and $\mathfrak{p}_{\alpha} \oplus \mathfrak{k}_{\alpha} = \mathfrak{g}_{\alpha}  \oplus \mathfrak{g}_{- \alpha}$,
\item[iii.] $\mathfrak{p} = \mathfrak{a} + \sum_{\alpha \in \Lambda} \mathfrak{p}_{\alpha}$, $\mathfrak{k} = \mathfrak{k}_0 + \sum_{\alpha \in \Lambda} \mathfrak{k}_{\alpha}$, where $\mathfrak{k}_0 = \mathfrak{g}_0 \cap \mathfrak{k}$.
\end{itemize}

For $X \in \mathfrak{a}$ we have that, along the geodesic $c$ in $M$ with initial conditions $c(0) = p$, $c'(0) = X$, the Jacobi fields are linear combinations of the following Jacobi fields:

$$c_{-\alpha(X)^2}(t)v_j(t) \textrm{ and } s_{-\alpha(X)^2}(t)v_j(t).$$

So, we define for a vector $X \in \mathfrak{a}$, and for $\alpha$ such that $\alpha(X) \neq 0$, the subspaces $P^u_{\alpha}(X), P^s_{\alpha}(X) \subset T_{(p,X)}UM$ such that

$$P^u_{\alpha}(X) = \{ (w,|\alpha(X)|w) \in \mathfrak{p} \times \mathfrak{p} : w \in \mathfrak{p}_{\alpha} \},$$ $$P^s_{\alpha}(X) = \{ (w,-|\alpha(X)|w) \in \mathfrak{p} \times \mathfrak{p} : w \in \mathfrak{p}_{\alpha} \},$$ where $T_{(p,X)}UM$ is identified with $T_pM \times T_pM$, which is identified with $\mathfrak{p} \times \mathfrak{p}$.

If follows from the definition that they are invariant by the geodesic flow.

Along the same lines of the proof that product metrics are not partially hyperbolic:

\begin{theorem}
\label{t.symm}
If the geodesic flow of a compact locally symmetric space of nonpositive curvature is partially hyperbolic, then it is a locally symmetric space of nonconstant negative curvature.
\end{theorem}
\begin{proof}
If the locally symmetric space $N$ has a partially hyperbolic geodesic flow, then the symmetric space $M$ such that $N = M / \Gamma$ has a partially hyperbolic geodesic flow.

Fix $x \in M$ and consider $v \in S_xM$. Let $\mathfrak{a}$ be the maximal Abelian subspace of $\mathfrak{g}$ in $x$ such that $v \in \mathfrak{a}$, after identification of $T_xM$ and $\mathfrak{p}$. 

Suppose $dim(\mathfrak{a}) \geq 2$. If the geodesic flow of the symmetric space $M$ is partially hyperbolic, then there is a splitting into invariant subbundles:
$$S(UM) = E^s \oplus E^c \oplus E^u.$$
\noindent This decomposition and the curvature tensor formula imply that
$$E^u(x,v) = \{ (\xi,\eta) \in T_{(x,v)}UM : (\xi,\eta) \in P^u_{\alpha_i}(v) \},$$
$$E^s(x,v) = \{ (\xi,\eta) \in T_{(x,v)}UM : (\xi,\eta)  \in P^s_{\alpha_i}(v) \},$$
\noindent $i = 1, \ldots, k$, $|\alpha_1(v)| > |\alpha_2(v)| > \ldots > |\alpha_k(v)|$, such that if $\beta \neq \alpha_i$, $\forall i = 1,\ldots,k$, then $\beta(v) < \alpha_i(v)$, $\forall i = 1,\ldots,k$.

Now we pick $(x,v')$ such that $\alpha_1(v') = 0$. Then:
$$E^u(x,v') = \{ (\xi,\eta) \in T_{(x,v')}UM : (\xi,\eta) \in P_{\beta_j} \},$$
$$E^s(x,v') = \{ (\xi,\eta) \in T_{(x,v')}UM : (\xi,\eta) \in P_{\beta_j} \},$$
\noindent for some $\beta_j \in \Lambda$, $j=1,\ldots,k'$, $|\beta_1(v')| > |\beta_2(v')| > \ldots > |\beta_{k'}(v')|$. Notice that $\alpha_1(v') = 0$ implies $\beta_j \neq \alpha_1$, $\forall j=1,\ldots,k'$. As in the proof of the product metric, there is no way to go from one decomposition to the other continuously. So, there are no Abelian subspaces with dimension greater than one, and the symmetric space of nonpositive curvature has rank one. If dimension of the Abelian subspaces is one then the symmetric space has negative curvature, which implies by the classification of Cartan \cite{He}, \cite{Hel} that it is a K\"ahler hyperbolic space, or quaternionic hyperbolic space, or the hyperbolic space over the Cayley numbers.
\end{proof}

\section{Further results and questions}
\label{s.further}

This section is about the obstructions to have a partially hyperbolic geodesic flow and some questions related to  partially hyperbolic geodesic flows.

There is an obstruction to partial hyperbolicity if one add the hypothesis of nonpositive sectional curvature in the Riemannian manifold: the rank of the Riemannian manifold.

\begin{definition}
Let $(M,g)$ be a Riemannian manifold of nonpositive sectional curvature. Then, for $v \in T_xM$, $rank(v) := \dim \mathcal{J}^c(v)$, where $\mathcal{J}^c(v)$ is the set of parallel Jacobi fields along the geodesic $\gamma$, such that $\gamma(0) = x$, $\gamma'(0) = v$. The rank of $M$ is $rank(M) := \inf_{v \in T_xM} rank(v)$ \cite{Ba1}, \cite{E2}, \cite{E3}.
\end{definition}

\begin{theorem}
\label{t.rankone}
If $M$ is a compact Riemannian manifold with nonpositive curvature such that its geodesic flow is partially hyperbolic, then $M$ has rank one.
\end{theorem}

\begin{proof}
By theorem \ref{t.prod}, $M$ has to be irreducible. By the rank rigidity theorem of Ballmann \cite{Ba2} and Burns-Spatzier \cite{BS}, if $M$ is irreducible, has nonpositive curvature, and rank bigger than one, then it is a locally symmetric space of rank bigger than one. Then, by theorem \ref{t.symm}, its geodesic flow is not partially hyperbolic.
\end{proof}

Another obstruction is the dimension of the Riemannian manifold, and also the dimension of the extremal subbundles of the partially hyperbolic splitting. We use the following result in Steenrod's classical book:

\begin{theorem}[\cite{St} Theorem 27.18]
\label{t.st}
Let $S^n$ be the n-dimensional sphere. Then, it does not admit a continuous field of tangent $k$-planes if $n$ is odd or if $n \equiv 1 \mod 4$ and $2 \leq k \leq n-2$.
\end{theorem}

So, we can state the following:

\begin{theorem}
\label{t.dimen}
If $(M^n,g)$ is a Riemannian manifold with partially hyperbolic geodesic flow then $n$ is even, and if $n \equiv 2 \mod 4$, then $\dim E^s = 1,n-2$ or $n-1$.
\end{theorem}

\begin{proof}
First, let $E^s \oplus E^c \oplus E^u$ be the splitting of $S(UM)$, the contact structure on the unit tangent bundle $UM$. Let $p \in M$ be fixed, and $U_pM$ the fiber of the unit tangent bundle of $M$ at $p$. Let $k := dim(E^s)$, $\Lambda(k,S_v(UM))$ be the set of $k$-dimensional isotropic subspaces of the symplectic space $S_v(UM)$, for any $v \in UM$, and $\Lambda_1(k,S_v(UM))$ be the set of $k$-dimensional isotropic subspaces of the symplectic space $S_v(UM)$ which intersect the vertical subspace, i.e., $E \in \Lambda_1(k,S_v(UM))$ if $E \cap V \neq \emptyset$. If we look at the fiber bundle $\pi: F_p \to U_pM$ whose fiber at $v \in U_pM$ is $\Lambda(k,S_v(UM))$. We know that the codimension of $\Lambda_1(k,S_v(UM))$ in $\Lambda(k,S_v(UM))$ is greater than one \cite{P} if $k < n-1$. So, if we look at $E^s(p,v)$ as a section of the fiber bundle $\pi: F_p \to U_pM$, it is easy to see that there is a section $\sigma: U_pM \to F_p$ which does not intersect the vertical subbundle. Let $\pi_S: S(UM) \to UM$ be the contact structure bundle on $UM$, then, $\pi_S(\sigma)$ is a continuous field of tangent $k$-planes in $U_pM$, which is a $(n-1)$-dimensional sphere. So, we need only to apply the previous theorem.

The case $k=n-1$ is trivial, there is no obstruction to the existence of a continuous field of $(n-1)$-planes on a $(n-1)$-dimensional sphere.

\end{proof}

\begin{remark}
The idea that partial hyperbolicity of the geodesic flow implies odd dimension of the Riemannian manifold is due to Gonzalo Contreras, who communicated an idea of the proof of this fact to the second author of this article. 
\end{remark}

There are some questions that we hope to address in the future:

\begin{question}
Is there a transitive partially hyperbolic non-Anosov geodesic flow with conjugate points?
\end{question}

It would be interesting to know if transitivity and existence of conjugate points can be together in these examples we constructed. For example, if the transitive non-Anosov example showed at section \ref{s.conc} is robustly transitive, the answer to the question would be positive.

\begin{question}
Is the example constructed in theorem A ergodic?
\end{question}

\end{document}